\documentclass[pdflatex,sn-basic,Numbered]{sn-jnl}

\usepackage{bbold}
\usepackage{graphicx}%
\usepackage{multirow}%
\usepackage{amsmath,amssymb,amsfonts}%
\usepackage{amsthm}%
\usepackage[title]{appendix}%
\usepackage{xcolor}%
\usepackage{textcomp}%
\usepackage{manyfoot}%
\usepackage{booktabs}%
\usepackage{algorithm}%
\usepackage{algorithmicx}%
\usepackage{algpseudocode}%
\usepackage{listings}%
\usepackage{mathtools}
\usepackage{enumitem}
\usepackage{lineno}

\newcommand{\wdp}{\ensuremath{w_p}}
\newcommand{\owp}{\ensuremath{\theta_p}}
\newcommand{\owpn}{\ensuremath{\theta_{p,n}}}

\newcommand{\rdowp}{\ensuremath{d^A_{p,n}}}
\newcommand{\dowpq}{\ensuremath{\rho_{p,q}}}
\newcommand{\real}{\ensuremath{\mathbb{R}}}

\newcommand{\borel}[1]{\ensuremath{P(#1)}}
\newcommand{\wasp}[1]{\ensuremath{W_p(#1)}}

\newcommand{\obs}[1]{\ensuremath{\Lambda(#1)}}

\newcommand{\otr}[1]{\ensuremath{T_{#1}}}

\newcommand{\empiric}[1]{\ensuremath{E(#1)}}
\newcommand{\etr}[1]{\ensuremath{D_#1}}
\newcommand{\atr}[1]{\ensuremath{T^n_{#1,A}}}
\newcommand{\mdim}[1]{\ensuremath{\xi_{#1}}}

\theoremstyle{thmstyleone}%
\newtheorem{theorem}{Theorem}
\newtheorem{proposition}[theorem]{Proposition}%
\newtheorem{lemma}[theorem]{Lemma}%
\newtheorem{corollary}[theorem]{Corollary}%

\theoremstyle{thmstyletwo}%
\newtheorem{example}{Example}%
\newtheorem{remark}{Remark}%

\theoremstyle{thmstylethree}%
\newtheorem{definition}{Definition}%

\begin{document}

\title[The Observable Wasserstein Distance]{The Observable Wasserstein Distance}


\author*[1]{\fnm{Edivaldo} \sur{Lopes dos Santos}}\email{edivaldo@ufscar.br}

\author[1]{\fnm{Leandro} \sur{Vicente Mauri}}\email{leandro.mauri@ufscar.br}

\author[2]{\fnm{Washington} \sur{Mio}}\email{wmio@fsu.edu}

\author[2]{\fnm{Tom} \sur{Needham}}\email{tneedham@fsu.edu}

\affil[1]{\orgdiv{Departmento de Matemática}, \orgname{Universidade Federal de São Carlos}, \country{Brazil}}

\affil[2]{\orgdiv{Department of Mathematics}, \orgname{Florida State University}, \country{USA}}


\abstract{We introduce the observable Wasserstein distance, a framework for deriving lower bounds on the Wasserstein distance between probability measures on Polish metric spaces, designed to bypass the computational intractability of exact optimal transport in large-scale, non-Euclidean datasets. Analogous to the sliced Wasserstein distance in $\mathbb{R}^d$, our approach projects measures onto the real line via 1-Lipschitz observables and computes the Wasserstein distances between the resulting pushforward distributions. We define a hierarchy of pseudo-metrics by restricting observables to a nested chain of subspaces. A central theoretical contribution is an injectivity result linking the metric covering dimension of the support of a measure to the specific order in the hierarchy that guarantees unique recovery. This serves as a metric-space analogue to the Cram\'{e}r-Wold Device for Euclidean distributions. We demonstrate that this hierarchy offers a tunable trade-off between sharpness as a lower bound on the Wasserstein distance and computational efficiency. We also present a discrete computational model for finite grids and numerical experiments validating the efficacy and utility of these approximations.}




\maketitle

\section{Introduction}

The primary objective of this paper is to develop the concept of the {\em observable Wasserstein distance} for probability measures on metric spaces, alongside a framework for computing the distance between data clouds within such domains. This approach follows a principle analogous to the sliced Wasserstein distance for probability distributions on Euclidean spaces \cite{rabin2011wasserstein}. Many datasets of interest consist of non-vector objects, such as 3D shapes, graphs, protein structures, and functional surfaces, which reside in metric spaces. Approaches like the sliced Wasserstein distance become insufficient in these contexts, and the observable Wasserstein distance provides a pathway to a computationally tractable alternative for deriving lower bounds on the exact Wasserstein distance.

The theoretical foundations of optimal transport trace back to Monge's 1781 formulation \cite{monge1781} and Kantorovich's seminal relaxation in the 1940s \cite{kantorovitch1958}. However, practical large-scale computation remained elusive until relatively recently. The computational landscape has experienced a transformative shift with the emergence of the entropic regularization paradigm of the Sinkhorn algorithm \cite{cuturi2013sinkhorn}, low-rank approximation methods \cite{lin2019efficient,scetbon2021lowrank}, and the sliced Wasserstein distance, which are pivotal mechanisms in large-scale data regimes. Although the Sinkhorn algorithm successfully reduced the historical cubic complexity of exact optimal transport, it remains bounded by quadratic computational demands relative to the number of data points, limiting its efficacy in massive datasets and motivating low-rank techniques. The sliced Wasserstein approach circumvents some of these bottlenecks by projecting Euclidean distributions onto one-dimensional subspaces where the transport problem admits an exact, closed-form solution via simple sorting, thereby achieving near-linear computational complexity. However, a critical determinant of efficiency is the number of projections required to ensure approximation accuracy, which can grow rapidly with the dimension of the ambient space. Despite this trade-off, the method is well-suited for applications in relatively low dimensions, in part because the computational workload is embarrassingly parallelizable. In particular, the sliced Wasserstein distance is effective for computations involving large point clouds in 3D space, where its ability to balance geometric fidelity with computational tractability offers a distinct advantage. The fact that the sliced Wasserstein distance is a metric, not just a pseudo-metric, is guaranteed by the Cram\'{e}r-Wold Device that states that a probability measure on $\real^d$ is uniquely determined by its projections onto 1-dimensional subspaces \cite{cramer1936}.

In a metric space $(X,d)$, where linearity, inner products, and norms are absent, our approach projects a probability distribution $\mu$ onto the real line using metric observables. These are  defined as 1-Lipschitz functions $f \colon X \to \real$, which act as non-expansive scalar fields: 
\begin{equation}
|f(x)-f(y)| \leq d(x,y),
\end{equation}
for any $x,y \in X$. The guiding principle is that, although the projected measure $f_\sharp \mu$ (for each observable $f$) only weakly retains information about the shape of $\mu$, the collection of these projections fully characterizes the probability measure in the aggregate. (The projection $f_\sharp \mu$ is defined by $f_\sharp \mu (A) = \mu (f^{-1} (A))$, for any Borel set $A \subseteq \real$.) Let $\Lambda (X)$ be the set of all metric observables $f \colon X \to \real$ endowed with the topology of uniform convergence on compact sets. Under suitable assumptions on the moments of the probability measures, the observable Wasserstein $p$-distance, $p\geq 1$, is defined as
\begin{equation} \label{E:owp}
\owp (\mu,\nu) \coloneqq \sup_{f \in \Lambda (X)} w_p (f_\sharp \mu, f_\sharp \nu),   
\end{equation}
where $w_p$ denotes the Wasserstein $p$-distance (see Definition \ref{D:wass}). Since observables are 1-Lipschitz functions, it readily follows that $\owp (\mu,\nu) \leq w_p (\mu, \nu)$. However, estimating the supremum in \eqref{E:owp} over all observables can be computationally expensive. This motivates the introduction of various subspaces of $\Lambda(X)$ that exhibit a strong connection with the shape of distributions on $X$. We begin with a subspace $\Lambda_\infty(X) \subseteq \Lambda(X)$ that suffices to determine any probability measure $\mu$ via its projections $f_\sharp \mu$ for $f \in \Lambda_\infty(X)$. The basic observables in $\Lambda_\infty(X)$ are the distance-to-a-point functions $f_a \colon X \to \mathbb{R}$ given by $f_a(x) = d(x,a)$, $a \in X$. For any $r > 0$, let $B(a,r)$ denote the open ball of radius $r$ centered at $a \in X$. Then,
\begin{equation} \label{E:mball}
{f_a}_\sharp \mu ([0,r)) = \mu (f^{-1}_a ([0,r))) =\mu (B(a,r)).
\end{equation}
Therefore, from the projections ${f_a}_\sharp \mu$ we can recover the measure of all open balls in $X$. The remaining elements $f \in \Lambda_\infty (X)$ are constructed from finite (weighted) wedge products of observables of the basic $f_a$ type so as to ensure that we can recover the measure of arbitrary finite unions of open balls from the projections $f_\sharp \mu$. (The wedge product of $f$ and $g$ is defined as $(f\wedge g) (x) = \min \{f(x), g(x)\}$.) This implies that the projections of any probability measure $\mu$ via the observables in $\Lambda_\infty (X)$ uniquely characterize $\mu$ so long as $(X,d)$ is a Polish metric space, an assumption that we make throughout. This is proven in Theorem \ref{T:injective}.

The construction of $\Lambda_\infty(X)$ naturally gives rise to an ascending chain of subspaces $\Lambda_n(X)$ $ \subseteq \Lambda_\infty(X)$, consisting of observables expressible as weighted wedge products of at most $n+1$ functions of the form $f_a$ (with $n \geq 0$). A question arises: does there exist a corresponding chain of subspaces of probability measures on $X$ that are fully characterized by their projections onto the real line via observables in $\Lambda_n(X)$? To address this, we introduce the notion of the \emph{metric covering dimension} of the support of a measure $\mu$. We answer the question in the affirmative, proving that for any $n \geq 0$, a probability measure $\mu$ whose support has dimension $\leq n$ is uniquely determined by its projections $f_\sharp \mu$ for $f \in \Lambda_n(X)$. In particular, if $X$ itself has metric covering dimension $\leq n$, then {\em any} probability measure on $X$ is recoverable from these projections. This injectivity result is established in Section \ref{S:ndim} and serves as a metric-space counterpart to the Cram\'{e}r--Wold Theorem for probability measures on Euclidean spaces.

For $n \geq 0$, by restricting the space of observables to $\Lambda_n(X)$, we define the pseudo-metrics
\begin{equation}
\theta_{p,n}(\mu, \nu) \coloneqq \sup_{f \in \Lambda_n(X)} w_p(f_\sharp \mu, f_\sharp \nu).
\end{equation}
The injectivity results discussed above imply that $\theta_{p,n}$ becomes a genuine metric when restricted to probability measures whose supports have metric covering dimension $\leq n$. Since $\Lambda_n(X) \subseteq \Lambda_m(X)$ for $n \leq m$, the following monotonicity holds:
\begin{equation}
\theta_{p,n}(\mu, \nu) \leq \theta_{p,m}(\mu, \nu) \leq w_p(\mu, \nu).
\end{equation}
As anticipated, there is a trade-off between the sharpness of $\theta_{p,n}(\mu, \nu)$ as a lower bound on the Wasserstein distance $w_p(\mu, \nu)$ and the computational efficiency gained by working with a lower-order space of observables. 

The connection between 1-Lipschitz observables and probability measures on metric spaces traces its roots to the celebrated Kantorovich-Rubinstein duality theorem for the Wasserstein 1-distance \cite{kantorovitch1958} (see also \cite{dudley2002}). This duality expresses $w_1$ in terms of the supremum of differences in expected values over all 1-Lipschitz functions, implying that the collection of such expectations completely characterizes a probability measure. Later, bounded 1-Lipschitz observables were employed to metrize weak convergence of probability measures on separable metric spaces using expected values \cite{dudley1966weak}. Our formulation of injectivity diverges from these expectation-based approaches by relying on the full pushforward measures $f_\sharp \mu$ rather than merely their means. While this may initially appear to be overly stringent, it proves essential. As illustrated by the injectivity argument for $\Lambda_\infty(X)$ outlined above, retaining the full distribution of observables allows us to identify an entire chain of metric observables that more explicitly capture the underlying geometry of the distributions. This distinction is not simply theoretical; it is critical for estimating observable Wasserstein distances in practice.

As noted above, the proposed framework is computationally efficient. Furthermore, empirical estimation of $\theta_{p,n}$ is straightforward to implement and inherently parallelizable. Our numerical experiments in Section~\ref{S:numexp} provide empirical evidence of the trade-off between computational efficiency and approximation accuracy. In practice, we find that a small number of observables suffices for the observable Wasserstein framework to outperform standard methods in various classification tasks. Additionally, we provide a proof-of-concept demonstration of how the observable Wasserstein distance can be integrated into deep learning pipelines.

\textbf{Organization} 
Section \ref{S:lip} introduces the Lipschitz transform $T_\mu$, which maps an observable $f$ to the pushforward measure $f_\sharp \mu$. It also defines the nested chain of observables $\Lambda_n (X)$ and establishes the injectivity result for $\Lambda_\infty (X)$. Section \ref{S:ndim} introduces the notion of metric covering dimension and demonstrates that the restriction of the Lipschitz transform to $\Lambda_n (X)$ uniquely identifies any probability measure whose support has metric covering dimension $\leq n$. In Section \ref{S:odistances}, we define and study the properties of the hierarchy of observable Wasserstein distances. 
Finally, Section \ref{S:numexp} presents numerical experiments, and Section \ref{S:discussion} concludes the paper with a summary and further discussion.


\section{The Lipschitz Transform} \label{S:lip}

Throughout the paper we assume that $(X,d)$ is a Polish (complete and separable) metric space and denote by $\borel{X}$ the collection of all Borel probability measures on $X$. If $f \colon X \to Y$ is a (Borel measurable) map between metric spaces and $\mu \in \borel{X}$, the pushforward $f_\sharp \mu \in \borel{Y}$ is given by $f_\sharp \mu (U):= \mu(f^{-1}(U))$, 
for any Borel measurable set $U \subseteq Y$. For $p \geq 1$, let $\wasp{X} \subseteq \borel{X}$ be the subset comprising all probability measures with finite $p$-moments; that is, 
\begin{equation}
\int_X d^p(x,x_0) d\mu (x) < \infty,
\end{equation}
for some (and thus all) $x_0 \in X$. We equip $\wasp{X}$ with the Wasserstein $p$-distance $\wdp$, which is defined as follows. Given $\mu, \nu \in \borel{X}$, a {\em coupling} between $\mu$ and $\nu$ is a probability measure $h \in \borel{X \times X}$ with the property that its marginals are $\mu$ and $\nu$. More precisely, if $\pi_1, \pi_2 \colon X \times X \to X$ denote the projections onto the first and second components, respectively, then ${\pi_1}_\sharp h = \mu$ and ${\pi_2}_\sharp h = \nu$. The collection of  couplings is denoted $\Gamma(\mu,\nu)$. 
\begin{definition}[cf.\,\cite{villani2009}] \label{D:wass}
Let $\mu,\nu \in \wasp{X}$ and $p \geq 1$. The {\em Wasserstein $p$-distance} between $\mu$ and $\nu$ is given by
\[
w_p (\mu,\nu) := \inf_{h \in \Gamma (\mu,\nu)} \Big(\int_{X \times X} d^p(x,y) dh(x,y) \Big)^{1/p}.
\]
\end{definition}
We refer to the metric space $(\wasp{X}, w_p)$ as the Wasserstein $p$-space. Jensen's inequality implies that $W_q (X) \subseteq \wasp{X}$ and $w_p (\mu,\nu) \leq w_q(\mu,\nu)$, for any $1 \leq p \leq q$. 

\smallskip

A 1-Lipschitz function $f \colon X \to \real$ is a non-expansive function; that is, $|f(x)-f(y)| \leq d(x,y)$, $\forall x,y \in X$. We refer to a 1-Lipschitz function as a {\em metric observable} (or simply an {\em observable}) and denote the set of all metric observables by $\obs{X}$.

\begin{example}  \label{E:dist}
For any $a \in X$, the distance-to-$a$ function $f_a \colon X \to \real$ given by $f_a (x) = d(x,a)$ is 1-Lipschitz, a fact that follows from the triangle inequality. The inclusion $\imath_X \colon X \hookrightarrow \obs{X}$ given by $\imath_X (a) = f_a$ lets us identify $X$ with a subspace of $\obs{X}$ through distance functions. Although the observables $f_a$ are not necessarily bounded functions (as $X$ can have infinite diameter), we have
\begin{equation}
\sup_{x \in X} |f_a (x)-f_b (x)| = d(a,b) \quad \forall a,b \in X.
\end{equation}

\end{example}

We equip $\obs{X}$ with the topology of uniform convergence on compact sets. If $X$ is compact, this topology is metrized by the $\|\cdot\|_\infty$ norm and $(\obs{X}, \|\cdot\|_\infty)$ is compact by the Arzel\`{a}-Ascoli Theorem (cf. \cite{dudley2002}).


\begin{definition}
Let $\mu \in \borel{X}$. The {\em Lipschitz transform} $\otr{\mu} \colon \obs{X} \to \borel{\real}$ is defined by
\[
f \mapsto \otr{\mu} (f) \coloneqq f_\sharp \mu,
\]
the pushforward of $\mu$ to $\real$ under the observable $f$. 
\end{definition}

If $\mu \in \wasp{X} \subseteq \borel{X}$, then $\otr{\mu} (f) \in \wasp{\real}$, $\forall f \in \obs{X}$. Indeed, let $x_0 \in X$ and $t_0=f(x_0) \in \real$. Then,
\begin{equation}
\int_\real |t-t_0|^p d f_\sharp \mu(t) = \int_X |f(x)-f(x_0)|^p d\mu (x) \leq \int_X d^p(x,x_0) d\mu (x) < \infty.
\end{equation}
We abuse notation and also denote by $\otr{\mu} \colon \obs{X} \to \wasp{\real}$ the Lipschitz transform viewed as having co-domain $\wasp{\real}$.

\subsection{Continuity}

We now establish continuity of the Lipschitz transform.

\begin{proposition} \label{P:liptr}
For any $\mu \in \wasp{X}$, $p \geq 1$, the Lipschitz transform $\otr{\mu} \colon \obs{X} \to \wasp{\real}$ is continuous. Moreover, if $X$ is compact, then $\otr{\mu} \colon (\obs{X}, \|\cdot\|_\infty) \to (\wasp{\real},w_p)$ is 1-Lipschitz.
\end{proposition}

The proof uses the following lemma.

\begin{lemma} \label{L:offk}
Let $\mu \in \wasp{X}$ and $f_n, f \in \obs{X}$, $n \geq 1$. Given $\epsilon >0$, if there is a point $x_0 \in X$ for which $f_n (x_0) \to f(x_0)$, as $n \to \infty$, then there exists a compact set $K \subseteq X$ and $n_0 \in \mathbb{N}$ such that, for any $n \geq n_0$, 
\[
\int_{K^c} |f_n(x) - f(x)|^p d\mu(x) < \epsilon^p.
\]
\end{lemma}

\begin{proof}
Since $X$ is Polish, $\mu$ is a tight measure \cite{dudley2002}, which means that for each $\delta>0$, there exists a compact set $L \subseteq X$ satisfying $\mu (L^c) < \delta$. This implies that there is a compact set $K \subseteq X$ such that 
\begin{equation} \label{E:moment}
\int_{K^c} d^p (x,x_0) d\mu(x) < (\epsilon/3)^p.
\end{equation}
Indeed, for each $m \geq 1$, let $A_m \subseteq X$ be a compact set satisfying $\mu (A^c_m) < 1/m$. The finite unions $K_m = \bigcup_{i=1}^m A_m$ form a chain $K_1 \subseteq \cdots \subseteq K_m \subseteq \cdots$ of compact sets with $\mu(K_m) \geq \mu(A_m) >1-1/m$. The measurable set $K_\infty = \bigcup_{i\geq 1} K_i$ satisfies $\mu (K_\infty) \geq \mu(K_m) > 1 - 1/m$, for every $m \geq 1$, so that $\mu(K_\infty)=1$. By the monotone convergence theorem,
\begin{equation}
\begin{split}
\int_{K_m} d^p (x,x_0) d\mu(x) &= \int_X d^p (x,x_0) \mathbb{1}_{\!K_m} (x) d\mu(x) \longrightarrow
\int_X d^p (x,x_0) \mathbb{1}_{\!K_\infty} (x) d\mu(x) \\
&= \int_X d^p (x,x_0) d\mu(x),
\end{split}
\end{equation}
as $m \to \infty$. Here, $\mathbb{1}_{\!K_m}$ denotes the characteristic function of $K_m$. Hence, \eqref{E:moment} holds for $K=K_m$, provided that $m$ is sufficiently large. Take $n$ large enough so that $|f_n(x_0)-f(x_0)| < \epsilon/3$. Then, \eqref{E:moment} and the Minkowski inequality imply that
\begin{equation}
\begin{split}
\Big(&\int_{K^c} |f_n(x)-f(x)|^p d\mu (x) \Big)^{1/p} \leq \Big(\int_{K^c} |f_n(x)-f_n(x_0)|^p d\mu (x) \Big)^{1/p} +\\
&+ \Big(\int_{K^c} |f_n(x_0)-f(x_0)|^p d\mu (x) \Big)^{1/p} + \Big(\int_{K^c} |f(x)-f(x_0)|^p d\mu (x) \Big)^{1/p} \\
&\leq 2 \Big(\int_{K^c} d^p(x,x_0) d\mu (x) \Big)^{1/p} + \Big(\int_{K^c} |f_n(x_0)-f(x_0)|^p d\mu (x) \Big)^{1/p}
< \epsilon.
\end{split}
\end{equation}
\end{proof}

\begin{proof}[Proof of Proposition \ref{P:liptr}]

Given $f,g \in \obs{X}$, let $\phi \colon X \to \real \times \real$ be the map $\phi(x) = (f(x),g(x))$ and define $h_\phi \coloneqq \phi_\sharp \mu$. The marginals of $h_\phi$ are ${\pi_1}_\sharp h_\phi = f_\sharp \mu$ and ${\pi_2}_\sharp h_\phi = g_\sharp \mu$ so that $h_\phi$ is a coupling between $\otr{\mu}(f)$ and $\otr{\mu}(g)$ that satisfies
\begin{equation} \label{E:coupling}
\begin{split}
w_p (f_\sharp \mu, g_\sharp \mu) &= \inf_{h \in \Gamma(f_\sharp (\mu), g_\sharp (\mu))} \Big(\int_{\real \times \real} |s-t|^p dh (s,t) \Big)^{1/p} \\
&\leq \Big(\int_{\real \times \real} |s-t|^p dh_\phi (s,t) \Big)^{1/p} = 
\Big(\int_X |f(x)-g(x)|^p d\mu (x)\Big)^{1/p}.
\end{split}
\end{equation}
Thus, if $X$ is compact, we have that $w_p (f_\sharp \mu, g_\sharp \mu) \leq \|f-g\|_\infty$, 
showing that $\otr{\mu}$ is 1-Lipschitz. Now we verify the continuity of $\otr{\mu}$ for any Polish space $(X,d)$. Let $f, f_n \in \obs{X}$, $n \geq 1$, be such that $f_n \to f$ in the topology of uniform convergence on compact sets. Clearly, for any $x_0 \in X$, we have pointwise convergence $f_n (x_0) \to f(x_0)$. Hence, given $\epsilon>0$, Lemma \ref{L:offk} ensures the existence of a compact set $K \subseteq X$ and $n_0 \in \mathbb{N}$ such that
\begin{equation} \label{E:offk}
\int_{K^c} |f_n(x) - f(x)|^p d\mu(x) < \frac{\epsilon^p}{2},
\end{equation}
for every $n \geq n_0$, where $K^c$ denotes the complement of $K$. From \eqref{E:coupling} and \eqref{E:offk}, we obtain
\begin{equation}
\begin{split}
w_p^p ({f_n}_\sharp \mu, f_\sharp \mu) &\leq \int_X |f_n(x)-f(x)|^p d\mu (x) \\
&= \int_K |f_n(x)-f(x)|^p d\mu (x) + \int_{K^c} |f_n(x)-f(x)|^p d\mu (x) \\
&< \int_K |f_n(x)-f(x)|^p d\mu (x) + \frac{\epsilon^p}{2},
\end{split}
\end{equation}
for $n \geq n_0$. Uniform convergence on compact sets guarantees that, taking $n$ sufficiently large, we have $|f_n(x)-f(x)| < \epsilon/ \sqrt[p]{2}$, for any $x \in K$. Therefore, $w_p ({f_n}_\sharp \mu, f_\sharp \mu) < \epsilon$.
This proves the continuity of the Lipschitz transform.
\end{proof}


\subsection{Identifiability from the Lipschitz Transform} \label{S:injective}

We show that a probability measure $\mu$ can be fully recovered from its Lipschitz transform $\otr{\mu}$. More precisely, the mapping $\mu \mapsto \otr{\mu}$ is injective. As a matter of fact, we prove that to determine $\mu$ it suffices to know the transform on a subset $\Lambda_\infty (X) \subseteq \obs{X}$ of observables that we describe next.

Given $f,g \colon X \to \real$, let $f\wedge g \colon X \to \real$ denote the minimum of $f$ and $g$; that is, the function defined by $(f\wedge g) (x) = \min\{f(x),g(x)\}$. The next lemma states a well-known fact.

\begin{lemma} \label{L:min}
If $f,g \in \obs{X}$, then $f\wedge g \in \obs{X}$.
\end{lemma}

\begin{proof}
We need to verify that $|(f\wedge g) (x) - (f\wedge g) (y)| \leq d(x,y)$, for any $x,y \in X$. If the minimum at both $x$ and $y$ is achieved by $f$ (or $g$), then the desired inequality follows from the fact that $f$ (or $g$) is 1-Lipschitz. Thus, without loss of generality, we may assume that $(f\wedge g) (x) = f(x)$ and $(f\wedge g) (y) = g(y)$. Then,
\begin{equation}
(f\wedge g) (x) - (f\wedge g) (y) = f(x) - g(y) \leq g(x) - g(y) \leq d(x,y).
\end{equation}
Similarly, $(f\wedge g) (y) - (f\wedge g) (x) \leq d(x,y)$. This proves the lemma.
\end{proof}

We now introduce an ascending chain
\begin{equation}
X \subseteq \Lambda_0(X) \subseteq \Lambda_1 (X) \subseteq \cdots \subseteq \Lambda_n (X) \subseteq \cdots \subseteq \Lambda_\infty (X) \subseteq\obs{X}
\end{equation}
of subspaces of $\obs{X}$ that are relevant to the injectivity argument. Recall that $I$ denotes the unit interval $[0,1]$. For $0 \leq n <\infty$, consider $(n+1)$-tuples $\alpha \in I^{n+1}$ and $a \in X^{n+1}$, written as $\alpha = (\alpha_0, \ldots, \alpha_n)$ and $a= (a_0, \ldots, a_n)$, and define
\begin{equation} \label{E:mmin}
f_a^\alpha \coloneqq \alpha_0 f_{a_0} \wedge \cdots \wedge \alpha_n f_{a_n}.    
\end{equation}
We refer to the points $a_0, \ldots, a_n \in X$ as the {\em anchor points} for $f_a^\alpha$, and to the scalars $\alpha_0, \ldots, \alpha_n \in I$ as the weights for $f_a^\alpha$. Clearly, $\alpha_i f_{a_i}$ is 1-Lipschitz for each $\alpha_i \in I$ and $a_i \in X$. Thus, by Lemma \ref{L:min}, $f_a^\alpha \in \Lambda(X)$. For $0 \leq n <\infty$, define
\begin{equation} \label{E:nchain}
\Lambda_n (X) = \{f_a^\alpha \colon \text{$\alpha \in I^{n+1}$ and $a \in X^{n+1}$}\}.    
\end{equation}
Since we do not require the points $a_i$ in \eqref{E:mmin} to be distinct, we have that $\Lambda_n (X) \subseteq \Lambda_{n+1} (X)$, for any $n \geq 0$. Moreover, as in Example \ref{E:dist}, we have an inclusion $X \hookrightarrow \Lambda_0 (X)$ given by $a \mapsto f_a$. We set
\begin{equation}
\Lambda_\infty (X) = \bigcup_{n=1}^\infty \Lambda_n (X).
\end{equation}
By construction, $\Lambda_n (X)$, $0 \leq n \leq \infty$, is closed under multiplication by scalars $\alpha \in I$. Moreover, $\Lambda_\infty (X)$ is also closed under the $\wedge$ operation.

To illustrate the role and utility of the observables in $\Lambda_n(X)$ in our injectivity argument, first consider distance functions $f_a \colon X \to \real$, $a \in X$. Let $B(a,r)$ denote the open ball of radius $r>0$ and center $a \in X$. Then, $B(a,r) = {f_a}^{-1}([0,r)) = {f_a}^{-1}((-\infty,r))$, where the last equality holds because $f_a \geq 0$. Therefore,
\begin{equation}
\mu (B(a,r)) = \mu (f^{-1}_a (-\infty,r)) = {f_a}_\sharp \mu ((-\infty,r))
\end{equation}
for any $\mu \in \borel{X}$, as pointed out in the Introduction. This implies that from the pushforward measures ${f_a}_\sharp \mu \in \borel{\real}$, $a \in X$, we can read off the $\mu$-mass of all open balls in $X$. More generally, let $a \in X^{n+1}$ and $\alpha = (1, \ldots, 1) \in I^{n+1}$. Then, the observable $f_a^\alpha = f_{a_0} \wedge \cdots \wedge f_{a_n}$ has the property that
\begin{equation} 
\bigcup_{i=0}^n B(a_i, r) = (f_a^\alpha)^{-1} (-\infty,r).    
\end{equation}
Therefore,
\begin{equation} \label{E:samerad}
\mu \left(\cup_{i=0}^n B(a_i, r) \right) = 
{f_a^\alpha}_\sharp \mu ((-\infty, r)),
\end{equation}
for any $\mu \in \borel{X}$. In other words, the $\mu$-mass of an arbitrary union of $n+1$ open balls of the same radius is determined by the observable measures ${f_a^\alpha}_\sharp \mu$. A more general form of \eqref{E:samerad} that allows balls of different radii is given in Lemma \ref{L:uniball} below. 

To visualize the functions $f_a^\alpha$, it may be helpful to consider the weighted Voronoi cells for the anchor points $a_i$, $0 \leq i \leq n$, which are given by
\begin{equation}
V_i \coloneqq \{x \in X \colon \alpha_i f_{a_i} (x)= f_a^\alpha (x)\}
= \{x \in X \colon \alpha_i d(x,a_i) \leq \alpha_j d(x, a_j), \text{for all $j\ne i$}\}.
\end{equation}
$V_i \subseteq X$ is the set of points whose closest anchor point is $a_i$ as measured by weighted distances. Figure \ref{fig:voronoi}(a) shows the level curves and weighted Voronoi cells for a randomly selected function $f \in \Lambda_4 (\real^2)$. The boundary of a particular sublevel set, which is the union of 5 open balls, is highlighted in blue. Figure \ref{fig:voronoi}(b) depicts the weighted Voronoi cells for a randomly chosen $f \in \Lambda_2 (V)$, where $V$ is the vertex set of a weighted graph equipped with the geodesic distance.
\begin{figure}
    \centering
    \begin{tabular}{cc}
    \begin{tabular}{c}
    \includegraphics[width=0.27\linewidth]{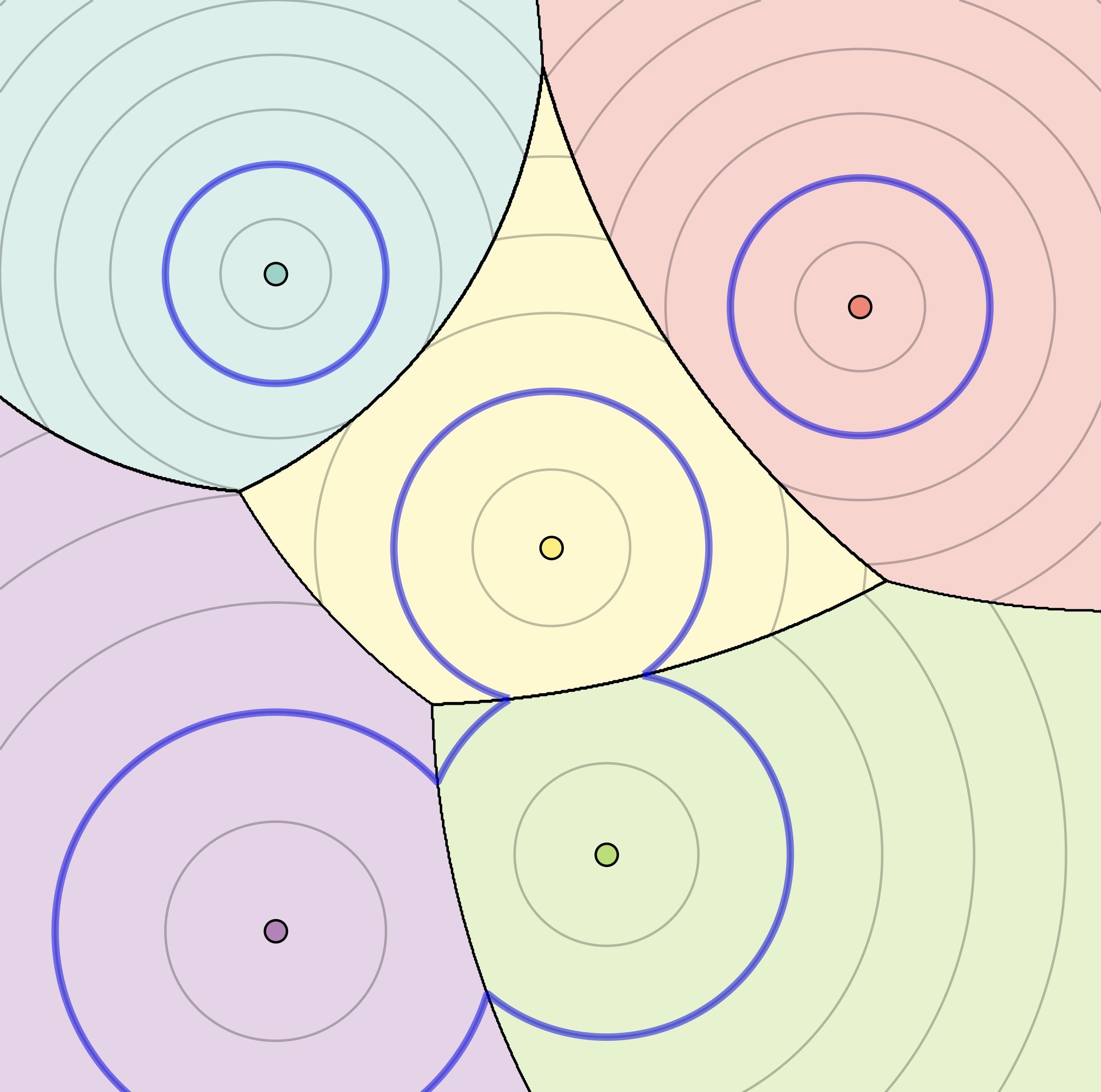}
    \end{tabular}
    \quad & \quad
    \begin{tabular}{c}
    \includegraphics[width=0.45\linewidth]{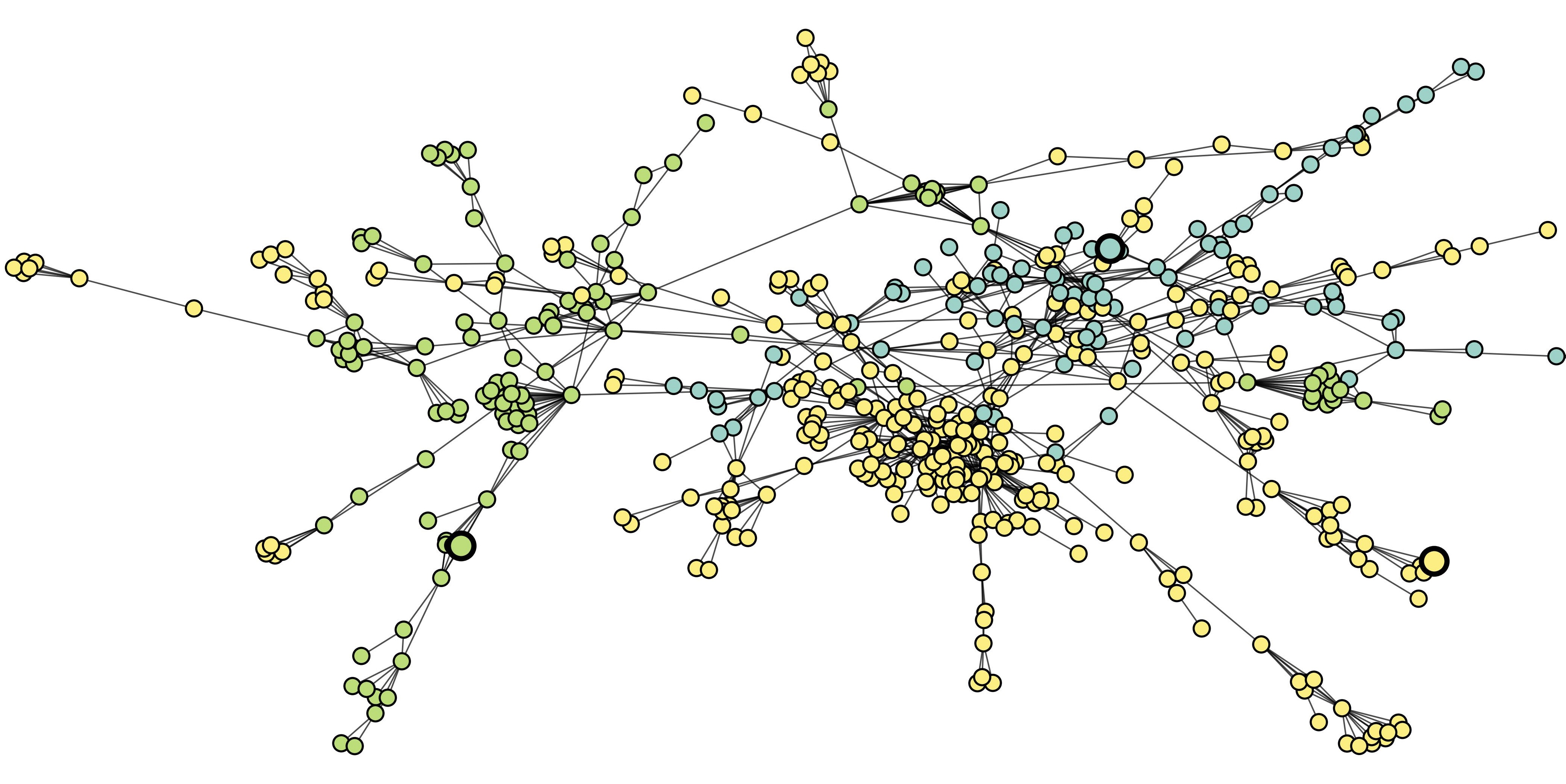}
    \end{tabular}
    \\
    (a) 5 anchor points & (b) 3 anchor points
    \end{tabular}
    \caption{(a) The level sets of a randomly selected function $f \in \Lambda_4(\real^2)$ and the respective weighted Voronoi cells and (b) the weighted Voronoi cells for a random $f \in \Lambda_2 (V)$, where $V$ is the vertex set of a weighted graph equipped with the shortest path distance.}
    \label{fig:voronoi}
\end{figure}

\begin{definition}
Let $\mu \in \borel{X}$ and $0 \leq n \leq \infty$. The {\em $n$-distance transform} $\otr{\mu}^n \colon \Lambda_n (X) \to \borel{\real}$ is the restriction of $\otr{\mu} \colon \obs{X} \to \borel{\real}$ to $\Lambda_n (X)$. We denote the restriction of $\otr{\mu}$ to $X \subseteq \Lambda (X)$ by $\etr{\mu} \colon X \to \borel{\real}$ and refer to it simply as the {\em distance transform}.
\end{definition}

As above, for $\mu \in \wasp{X}$, the transform $\otr{\mu}^n$ may be viewed as having co-domain $\wasp{\real}$.

\begin{theorem}[Injectivity] \label{T:injective}
If $\mu,\nu \in \borel{X}$ and $\otr{\mu}^\infty = \otr{\nu}^\infty$, then $\mu=\nu$. In particular, $\otr{\mu} = \otr{\nu}$ if and only if $\otr{\mu}^\infty = \otr{\nu}^\infty$.
\end{theorem}

Lemma \ref{L:union} and Proposition \ref{P:chain} below are used in the proof of injectivity.

\begin{lemma} \label{L:union}
Let $f,g \colon X \to \real$. If $U = f^{-1} (-\infty,r)$ and $V = g^{-1} (-\infty, r')$, with $0 <r \leq r'$, then $h = f \wedge (\alpha g)$ satisfies $U \cup V = h^{-1} (-\infty,r)$, where $\alpha = r/r'$.
\end{lemma}

\begin{proof}
Let $\alpha = r/r' \leq 1$ and $h = f \wedge (\alpha g)$. Since $g^{-1} (-\infty, r') = (\alpha g)^{-1} (-\infty, r)$, we have
\begin{equation}
h^{-1} (-\infty,r) = \{x \in X \colon \text{$f(x) < r$ or $\alpha g(x) < r$}\}
= \{x \in X \colon \text{$f(x) < r$ or $g(x) < r'$}\} = U \cup V ,
\end{equation}
as claimed.
\end{proof}

\begin{proposition} \label{P:chain}
If  $U \subseteq X$ is an open set, then there exists a chain $U_1 \subseteq U_2 \subseteq \cdots \subseteq U_j \subseteq \cdots$ of open sets such that $U= \bigcup_{j=1}^\infty U_j$ and each $U_j$ can be expressed as $U_j = f_j^{-1} ((-\infty,r_j))$ with $f_j \in \Lambda_\infty (X)$ and $r_j>0$.
\end{proposition}

\begin{proof}
Since $X$ is separable, we can write $U =\bigcup_{i=1}^{\infty} B(a_i, \epsilon_i)$, a countable union of open balls centered at $a_i \in X$ with radius $\epsilon_i>0$. Let $U_j = \bigcup_{i=1}^{j} B(a_i,\epsilon_i)$. Clearly $U_j \subseteq U_{j+1}$, for each $j \geq 1$, and $U= \bigcup_{j=1}^\infty U_j$. We construct the sequences $(f_j)_{j=1}^\infty$ and $(r_j)_{j=1}^\infty$ by induction.

Since $U_1 = B(a_1,\epsilon_1)$, we take $f_1= f_{a_1}$ and $r_1=\epsilon_1$. Then, $U_1=B(a_1, \epsilon_1) = f_{a_1}^{-1}(-\infty, r_1) $ $= f_1^{-1}(-\infty, r_1)$. For the inductive step, suppose $f_n \in \Lambda_\infty (X)$ and $r_n>0$ have been constructed such that $U_n = f_n^{-1}(-\infty,r_n)$. Let $r_{n+1} = \min \{r_n, \epsilon_{n+1}\} >0$. Since $B(a_{n+1}, \epsilon_{n+1}) = f_{a_{n+1}}^{-1} (-\infty, \epsilon_{n+1})$, Lemma \ref{L:union} ensures that there exists $f_{n+1} \in \Lambda_\infty (X)$ such that 
\begin{equation}
U_{n+1} = U_n \cup B(a_{n+1}, \epsilon_{n+1}) = f_{n+1}^{-1} (-\infty, r_{n+1}).
\end{equation}
This completes the proof.
\end{proof}

\begin{proof}[Proof of Theorem \ref{T:injective}]

Since the Borel $\sigma$-algebra is generated by the open sets of $X$, it suffices to show that $\mu(U)=\nu(U)$, for any open set $U \subseteq X$. By Proposition \ref{P:chain}, there is a chain $U_1 \subseteq U_2 \subseteq \cdots \subseteq U_j \subseteq \cdots$ of open sets such that $U= \bigcup_{j=1}^\infty U_j$ and each $U_j$ is of the form $U_j = f_j^{-1} ((-\infty,r_j))$, with $f_j \in \Lambda_\infty (X)$ and $r_j>0$. This fact and the assumption that $\otr{\mu}^\infty = \otr{\nu}^\infty$ imply that, for each $j\geq 1$, we have 
\begin{equation}
\mu(U_j) = \mu \big(f_j^{-1} (-\infty,r_j)\big) = {f_j}_\sharp (\mu) (-\infty, r_j) = {f_j}_\sharp (\nu) (-\infty, r_j) 
= \nu \big(f_j^{-1} (-\infty,r_j)\big) = \nu(U_j).
\end{equation}
Therefore, $\mu(U) = \lim_{j \to \infty} \mu(U_j) = \lim_{j \to \infty} \nu(U_j) = \nu(U)$, as desired. The second statement is a direct consequence of the first.
\end{proof}

\section{Identifiability and Dimension} \label{S:ndim}

Parallel to the chain of observables $\Lambda_n (X)$, we introduce a chain $\Omega_n (X)$ of subspaces of $\borel{X}$,  which induce a stratification of $\borel{X}$ by measures whose supports satisfy a dimensionality constraint. The main result of this section is a stratified injectivity theorem that asserts that any probability measure $\mu \in \Omega_n(X)$ is determined by its $n$-distance transform $\otr{\mu}^n \colon \Lambda_n (X) \to \borel{\real}$.

For a closed subset $S \subseteq X$, we define an (extended) non-negative integer $\mdim{X} (S)$ that can be thought of as a covering dimension of $S$ by open balls in $X$. 

\begin{definition}
Let $S \subseteq X$ be a closed subset.
\begin{enumerate}[label=\rm{(\roman*)}]
\item $S \subseteq X$ has {\em covering dimension by balls} $\leq n$, if any cover of $S$ by open sets in $X$ can be refined to a cover of $S$ by open balls of $X$ such that no sub-collection of $k$ distinct balls in the refinement has non-empty intersection if $k> n+1$.
\item The {\em metric covering dimension} of $S$ in $X$ is defined as
\[
\mdim{X}(S) \coloneqq \inf \{n \colon \text{covering dimension of $S$ by balls} \leq n\}.
\]
If no such non-negative integer $n$ exists, set $\mdim{X} (S) \coloneqq \infty$. We adopt the abbreviation $\mdim{X} \coloneqq \mdim{X} (X)$ and refer to $\mdim{X}$ as the {\em metric covering dimension} of $X$. 
\end{enumerate}
\end{definition}

\begin{example}
    The most practically relevant setting is $S \subseteq X$, $|S| <\infty$. In this case, any open cover of $S$ can be refined to a cover consisting of pairwise disjoint balls, implying that $\mdim{X}(S) = 0$. For a less trivial example, suppose that $X = \mathbb{R}$, with its standard metric.  Clearly, the Lebesgue covering dimension of $X$ lower bounds $\mdim{X}$, so that $\mdim{\mathbb{R}} \geq 1$. Now let $\mathcal{U}$ be an open cover of $\mathbb{R}$ and choose an open refinement $\mathcal{V}$ such that no 3 distinct elements of $\mathcal{V}$ have nonempty intersection. Each set in $\mathcal{V}$ is a disjoint union of open intervals, which we can assume to be bounded without loss of generality (by refining further, as necessary). Let $\mathcal{W}$ be the collection of all distinct open intervals appearing in the elements of $\mathcal{V}$. Then $\mathcal{W}$ is a refinement of $\mathcal{U}$ to a cover by open balls, and it is simple to show that no three distinct balls can have nonempty intersection. It follows that $\mdim{\mathbb{R}} = 1$. 
\end{example}

\begin{remark}
If $\phi \colon X \to Y$ is an isometry, then $\mdim{X} (A) = \mdim{Y} (\phi(A))$ for all $A \subset X$. 
\end{remark}

\begin{definition}
For $0 \leq n < \infty$, define $\Omega_n (X) \coloneqq \{\mu \in \borel{X} \colon \mdim{X} (\text{supp}\,\mu) \leq n\}$, the set of probability measures whose supports have metric covering dimension $\leq n$.    
\end{definition}

The sets $\Omega_n (X) \subseteq \borel{X}$ form a chain
\begin{equation}
\Omega_0 (X) \subseteq \cdots \subseteq \Omega_n (X) \subseteq 
\cdots \subseteq \Omega_\infty (X) = \borel{X}
\end{equation}
and produce an associated stratification of $\borel{X}$ whose strata $P_n(X)$ are:
\begin{enumerate}[label=\rm{(\roman*)}]
\item $P_0 (X) \coloneqq \Omega_0 (X)$, comprising all probability measures with $0$-dimensional support;
\item $P_n (X) \coloneqq \Omega_n(X) \setminus \Omega_{n-1} (X), 1 \leq n <\infty$, the set of all probability measures whose supports have dimension exactly $n$;
\item $P_\infty(X) \coloneqq \borel{X} \setminus \cup_{n=0}^\infty \Omega_n (X)$, comprising all probability measures with infinite-dimensional support.
\end{enumerate}
Clearly, if the metric covering dimension of $X$ satisfies $\mdim{X} <\infty$, then this is a finite stratification because $P_n(X) = \emptyset$ for $n > \mdim{X}$.

\begin{example} \label{EX:empiric}
In practice, we are typically interested in empirical measures obtained from independent random samples of a measure $\mu$. These empirical measures are contained in the set $\empiric{X}$ of finitely supported measures; that is,
\begin{equation}
\empiric{X} :=\Big\{\sum_{i=1}^n \lambda_i \delta_{x_i} \colon
\text{$n \in \mathbb{N}$, $x_i \in X$, $\lambda_i>0$, and $\lambda_1 + \cdots + \lambda_n =1$} \Big\},
\end{equation}
where $\delta_x$ denotes the Dirac atom supported at $x \in X$. Clearly, $\empiric{X} \subseteq \Omega_0 (X)$.
\end{example}

The next goal is to obtain an analogue of Theorem \ref{T:injective} for $\otr{\mu}^n$, $0 \leq n <\infty$. The first step toward this objective is to express the measure of an arbitrary union or intersection of $n+1$ open balls in terms of  pushforward measures by observables in $\Lambda_n (X)$. We adopt the following notation. Given $n \geq 0$, and $0 \leq k \leq n$, let
\begin{equation}
I_k \coloneqq \{(i_0, \ldots, i_k) \colon 0 \leq i_0 < \cdots <i_k \leq n\}. 
\end{equation} 
Let $B(a_i, r_i) \subseteq X$, $a_i \in X$ and $r_i >0$, $0 \leq i \leq n$, be a collection of open balls in $X$. Without loss of generality, we can assume that the balls are ordered so that $r_0 \leq \ldots \leq r_n$. For $\imath=(i_0, \ldots, i_k) \in I_k$, let
\begin{equation} \label{E:hi}
h_\imath = f_{a_{i_0}} \wedge \left(\frac{r_{i_0}}{r_{i_1}}\right) f_{a_{i_1}} \wedge \cdots \wedge  \left(\frac{r_{i_0}}{r_{i_k}}\right) f_{a_{i_k}} \in \Lambda_k(X) \subseteq \Lambda_n (X).    
\end{equation}
For $n\geq 0$, we write $\imath_n = (0, \ldots, n)$, so that
\begin{equation} \label{E:hn}
h_{\imath_n} = f_{a_0} \wedge \left(\frac{r_0}{r_1}\right) f_{a_1} \wedge \cdots \wedge  \left(\frac{r_0}{r_n}\right) f_{a_n}.    
\end{equation}

\begin{lemma} \label{L:uniball}
The function $h_{\imath_n} \in \Lambda_n (X)$ satisfies:
\begin{enumerate}[label=\rm{(\roman*)}]
\item $\bigcup_{i=0}^n B(a_i,r_i) =  h_{\imath_n}^{-1} ((-\infty,r_0))$.
\item $\mu (\bigcup_{i=0}^n B(a_i,r_i)) = (h_{\imath_n})_\sharp \mu ((-\infty,r_0))$, for any $\mu \in \borel{X}$.
\end{enumerate}
\end{lemma}
\begin{proof}
To prove (i), recall that the balls are ordered by increasing radii. The argument is by induction on $n$. The lemma clearly holds for $n=0$ because $B(a_0,r_0) = f^{-1}_{a_0}((-\infty,r_0))$. For the inductive step, write
\begin{equation}
\bigcup_{i=0}^n B(a_i,r_i) = \Big(\bigcup_{i=0}^{n-1} B(a_i,r_i)\Big) \bigcup   B(a_n,r_n) = h^{-1}_{\imath_{n-1}} ((-\infty,r_0)) \cup f^{-1}_{a_n} ((-\infty, r_n)).
\end{equation}
Then, the claim follows from Lemma \ref{L:union}. Statement (ii) follows from (i) and the definition of the pushforward of a measure. 
\end{proof}

To simplify notation, for any $r>0$, we write $I_r = (-\infty, r)$. For $\imath = (i_0, \ldots, i_k)$, we let $I_\imath\coloneqq I_{r_{i_0}} = (-\infty, r_{i_0})$. 

\begin{lemma} \label{L:mint}
Let $B(a_i, r_i) \subseteq X$, $a_i \in X$ and $r_i >0$, $0 \leq i \leq n$, be a collection of open balls in $X$. Then, the equality 
\[
\mu(\cap_{i=0}^n B(a_i,r_i)) = \sum_{k=0}^n \sum_{\imath \in I_k} (-1)^k {h_\imath}_\sharp \mu (I_\imath)
\]
holds for all $\mu \in \borel{X}$, where $h_\imath \in \Lambda_k (X)$ is the function defined in \eqref{E:hi}.
\end{lemma}

\begin{proof}
As before, we assume that the balls $B(a_i,r_i)$ are ordered by non-decreasing radii. By the dual inclusion-exclusion principle \cite[Chapter 2]{Graham1994}, we have
\begin{equation} \label{E:d-incexc}
\mu\big(\cap_{i=0}^n B(a_i,r_i)\big) =  
\sum_{k=0}^n \sum_{\imath \in I_k}  (-1)^k \mu \big(B(a_{i_0}, r_{i_0}) \cup \cdots \cup (B(a_{i_k}, r_{i_k})\big).
\end{equation}
By Lemma \ref{L:uniball}(ii), for each $\imath \in I_k$, 
\begin{equation}
\mu \big(B(a_{i_0}, r_{i_0}) \cup \cdots \cup (B(a_{i_k}, r_{i_k})\big) = {h_\imath}_\sharp \mu (I_\imath).
\end{equation}
This proves the lemma.
\end{proof}

\begin{proposition} \label{P:union}
Let $B(a_i, r_i) \subseteq X$, $a_i \in X$ and $r_i >0$, $0 \leq i \leq N$, be a collection of $N+1$ pairwise distinct open balls in $X$, and let $1 \leq n \leq N$. If  
\[
B (a_{i_0},r_{i_0}) \cap \cdots  \cap B (a_{i_k},r_{i_k}) = \emptyset,
\]
for any $0 \leq i_0 < \cdots <i_k \leq N$ with $k>n$, then there exist an integer $n_0 >0$, a collection of functions $\phi_j \in \Lambda_n (X)$ and intervals $I_j \subseteq \real$, $1 \leq j \leq n_0$, such that 
\[
\mu\big(\cup_{i=0}^N B(a_i,r_i)\big) = \sum_{j=1}^{n_0} \epsilon_j
\phi_{j\sharp} \mu (I_j)
\]
for any $\mu \in \borel{X}$, where $\epsilon_j = \pm 1$.
\end{proposition}

\begin{proof}
Assume that the balls are ordered by non-decreasing radii. By the inclusion-exclusion principle,
\begin{equation} \label{E:incexc}
\begin{split}
\mu\big(\cup_{i=0}^N B(a_i,r_i)\big) &=  
\sum_{\ell=0}^N \sum_{\imath \in I_\ell}  (-1)^\ell \mu \big(B(a_{i_0}, r_{i_0}) \cap \cdots \cap (B(a_{i_\ell}, r_{i_\ell})\big) \\
&=  
\sum_{\ell=0}^n \sum_{\imath \in I_\ell} (-1)^\ell \mu \big(B(a_{i_0}, r_{i_0}) \cap \cdots \cap B(a_{i_\ell}, r_{i_\ell})\big),
\end{split}
\end{equation}
where the last equality holds because of the assumption on empty intersections of order larger than $n+1$. By Lemma \ref{L:mint}, for each $\imath \in I_\ell$, $0 \leq \ell \leq n$, we can write
\begin{equation} \label{E:gint}
\mu \big(B(a_{i_0}, r_{i_0}) \cap \cdots \cap B(a_{i_\ell}, r_{i_\ell})\big) 
= \sum_{k=0}^\ell \sum_{\imath \in I_k} (-1)^k {h_\imath}_\sharp \mu (I_\imath).
\end{equation}
The claim follows from \eqref{E:incexc} and \eqref{E:gint}.
\end{proof}

This prepares us to prove the main result of this section.

\begin{theorem}[Stratified Injectivity] \label{T:ninjective}
Let $\mu, \nu \in P_n (X)$, $0 \leq n < \infty$. Then, $\otr{\mu}^n = \otr{\nu}^n$ if and only if $\mu=\nu$. 
\end{theorem}

\begin{proof}
If $\mu=\nu$, then it is clear that $\otr{\mu}^n = \otr{\nu}^n$. For the converse statement, let $S_\mu$ and $S_\nu$ be the supports of $\mu$ and $\nu$, respectively. Since $(X,d)$ is Polish, by a theorem of Ulam \cite[Theorem 7.1.4]{dudley2002}, any probability measure on $X$ is regular. Therefore, to prove that $\mu=\nu$, it suffices to show that $\mu(K) = \nu (K)$, for any compact set $K \subseteq S_\mu$ or $K \subseteq S_\nu$. Suppose $K \subseteq S_\mu$ with $K$ compact. It is simple to verify that $\mdim{X} (S_\mu) \leq n$ implies that $\mdim{X} (K) \leq n$. 

Given an integer $m>0$, consider the cover of $K$ by all open balls $B(x,1/m)$ with $x \in K$. We adopt the notation
\begin{equation}
K_m \coloneqq \bigcup_{x \in K} B(x,1/m)
\end{equation}
for the $1/m$-thickening of $K$. By assumption, this cover can be refined to a cover by open balls in $X$ with the property that any $k$-fold intersection of distinct elements of the cover is empty if $k>n+1$. Since $K$ is compact, we can assume that this cover is finite and denote its elements by $B (x_{mi}, r_{mi})$, where $x_{mi} \in X$, $r_{mi} >0$, and $0 \leq i \leq N_m$, for some $N_m >0$. Set
\begin{equation}
L_m \coloneqq \bigcup_{i=0}^{N_m} B(x_{mi}, r_{mi}).    
\end{equation}
By construction $K \subseteq L_m \subseteq K_m$, for any $m \geq 1$. By Proposition \ref{P:union}, there are functions $\phi_{mj} \in \Lambda_n (X)$ and intervals of the form $I_{mj} = (-\infty, r_{mj})$, $1 \leq j \leq N_m$, such that 
\begin{equation}
\eta (L_m) = \sum_{j=1}^{N_m} \epsilon_{mj} \phi_{mj\sharp} \eta (I_{mj}),    
\end{equation}
for any $\eta \in \borel{X}$, where $\epsilon_{mj}=\pm 1$. Thus, the assumption that $\otr{\mu}^n = \otr{\nu}^n$ implies that $\mu(L_m) = \nu (L_m)$, for any $m \geq 1$. Since $K = \cap_{m=1}^\infty K_m$, we have that
\begin{equation} 
\mu(K) = \lim_{m \to \infty} \mu (K_m) \quad \text{and} \quad \nu(K) = \lim_{m \to \infty} \nu (K_m). 
\end{equation}
Using the fact that $K \subseteq L_m \subseteq K_m$, $\forall m\geq 1$, we can conclude that
\begin{equation}
\mu (K) = \lim_{m \to \infty} \mu (K_m) = \lim_{m \to \infty} \mu (L_m) 
\quad \text{and} \quad  
\nu (K) = \lim_{m \to \infty} \nu (K_m) = \lim_{m \to \infty} \nu (L_m).
\end{equation}
Thus,
\begin{equation}
\mu(K) = \lim_{m \to \infty} \mu(L_m) = \lim_{m \to \infty} \nu(L_m) = \nu (K),
\end{equation}
as claimed. The case $K \subseteq S_\nu$ is similar.
\end{proof}

\begin{corollary} \label{C:oinjective}
Let $\mu,\nu \in \empiric{X}$. Then, $\otr{\mu}^0 = \otr{\nu}^0$ if and only if $\mu=\nu$.
\end{corollary}
\begin{proof}
This follows from Theorem \ref{T:ninjective} and the fact that $\empiric{X} \subseteq \Omega_0 (X)$.
\end{proof}

\begin{corollary}
Suppose $(X,d)$ is a Polish metric space whose metric covering dimension satisfies $\mdim{X} \leq n$ and let $\mu, \nu \in \borel{X}$. Then, $\otr{\mu}^n = \otr{\nu}^n$ if and only if $\mu = \nu$.
\end{corollary}
\begin{proof}
For any $\eta \in \borel{X}$, $\mdim{X} \leq n$ implies that $\mdim{X} (S_\eta) \leq n$, where $S_\eta$ is the support of $\eta$. Therefore, $\Omega_n (X) = \borel{X}$. Thus, if $\otr{\mu}^n = \otr{\nu}^n$, Theorem \ref{T:ninjective} guarantees that $\mu=\nu$. The converse statement is trivial.   
\end{proof}

\begin{remark}
Theorem \ref{T:ninjective} also holds with the set of observables $\Lambda_n (X)$ in the definition of $\otr{\mu}^n$ replaced with $\hat{\Lambda}_n (X) \subseteq \obs{X}$, $n \geq 0$, consisting of observables of the form $g_a^\alpha \coloneqq \alpha_0 f_{a_0} \vee \cdots \vee \alpha_n f_{a_n}$, 
where $f\vee g$ denotes the maximum of $f$ and $g$. 
\end{remark}

\section{The Observable Wasserstein Distance} \label{S:odistances}

In this section, we develop various (pseudo) metrics for probability measures on a Polish space $(X,d)$, based on the observable transforms described above. To show that the definitions are well-posed, we invoke the following standard result. A proof appears, e.g., in~\cite[Lemma 4.11]{gomez2025metrics}.

\begin{lemma} \label{L:bound}
If $\mu, \nu \in \wasp{X}$ and $f \in \obs{X}$, then $w_p (f_\sharp\mu, f_\sharp \nu) \leq w_p(\mu, \nu)$.
\end{lemma}

\begin{definition}
Let $\mu, \nu \in \wasp{X}$, $p \geq 1$.
\begin{enumerate}[label=\rm{(\roman*)}]
\item The {\em observable Wasserstein $p$-distance} between $\mu$ and $\nu$ is defined as
\[
\owp (\mu, \nu) \coloneqq \sup_{f \in \obs{X}} w_p (\otr{\mu}(f), \otr{\nu} (f)) 
= \sup_{f \in \obs{X}} w_p (f_\sharp\mu, f_\sharp \nu).
\]
\item By restricting $\owp$ to $\Lambda_n (X)$, $0 \leq n \leq \infty$, define 
\[
\owpn (\mu, \nu) \coloneqq \sup_{f \in \Lambda_n(X)} w_p (\otr{\mu}^n(f), \otr{\nu}^n (f)) 
= \sup_{f \in \Lambda_n (X)} w_p (f_\sharp \mu , f_\sharp \nu).
\]
\end{enumerate}
\end{definition}
The finiteness of $\owp (\mu, \nu)$ and $\owpn (\mu, \nu)$ follows from Lemma \ref{L:bound} because $w_p (\mu, \nu)<\infty$. The distance functions $\owp \colon \wasp{X} \times \wasp{X} \to \real$ and $\owpn \colon \wasp{X} \times \wasp{X} \to \real$, $0 \leq n \leq \infty$, clearly define pseudo-metrics on $\wasp{X}$, as $w_p$ is a metric on $\wasp{\real}$. We refine the properties of these distance functions in the following proposition.

\begin{proposition} \label{P:distance}
Let $\Omega_{p,n}(X) = \Omega_n (X)\cap\wasp{X}$, where $p \geq 1$ and $0 \leq n \leq \infty$. The following statements hold:
\begin{enumerate}[label=\rm{(\roman*)}]
\item $\owp \colon \wasp{X} \times \wasp{X} \to \real$ is a metric;
\item the restriction of $\owpn$ to $\Omega_{p,n}(X)$, $0 \leq n \leq \infty$, defines a metric $\owpn \colon \Omega_{p,n}(X)  \times \Omega_{p,n}(X)  \to \real$;
\item if $0\leq m \leq n \leq \infty$, then
\[
\theta_{p,m} (\mu,\nu) \leq \owpn (\mu,\nu) \leq \owp (\mu,\nu) \leq w_p (\mu,\nu),
\]
for any $\mu, \nu \in \wasp{X}$.
\end{enumerate}
\end{proposition}
\begin{proof}
Theorem \ref{T:injective} implies that $\owp$ and $\theta_{p,\infty}$ are metrics, and the fact that $\owpn$, $0 \leq n <\infty$, is also a metric follows from Theorem \ref{T:ninjective}. The inequalities $\theta_{p,m} (\mu,\nu) \leq \owpn (\mu,\nu) \leq \owp (\mu,\nu)$ in (iii) follow from the fact that  $\Lambda_m (X) \subseteq \Lambda_n (X) \subseteq \obs{X}$. It remains to show that $\owp (\mu,\nu) \leq w_p (\mu,\nu)$. By Lemma \ref{L:bound}, $w_p (f_\sharp \mu, f_\sharp \nu) \leq w_p (\mu, \nu)$, for any $f \in \obs{X}$. Thus, the inequality carries over to the supremum over $f \in \obs{X}$; that is, $\owp (\mu,\nu) \leq w_p (\mu,\nu)$.
\end{proof}

\begin{remark}\label{rmk:LpVersion}
For spaces such as a compact Riemannian manifold $(M,g)$ equipped with the geodesic distance, we can define other versions of $\theta_{p,0}$ by averaging over $M$ using the volume measure instead of taking the supremum over $ \Lambda_0 (X)$. For example, for $p,q\geq 1$, we can define
\begin{equation}
\dowpq (\mu,\nu) = \Big(\int_M w_p^q ({f_x}_\sharp \mu, {f_x}_\sharp \nu) d\,Vol (x)\Big)^{1/q}.
\end{equation}
A version of this form is implemented numerically in Section~\ref{sec:autoencoder}. For the sake of keeping the exposition simple, we will continue to focus on the supremum version when developing theory.
\end{remark}


For $p=1$, the Kantorovich-Rubinstein duality in optimal transport implies that the inequality $\theta_1(\mu,\nu) \leq w_1(\mu, \nu)$ in Proposition \ref{P:distance} can be strengthened to an equality.

\begin{proposition}\label{prop:equivalence_p_equals_1}
    For $\mu,\nu \in W_1(X)$, $\theta_1(\mu,\nu) = w_1(\mu,\nu)$.
\end{proposition}

\begin{proof}
    By Kantorovich-Rubinstein duality (cf.~\cite[Theorem 11.8.2]{dudley2002}, \cite[Theorem 5.10 and Remark 6.5]{villani2009}), we can write 
    \[
    w_1(\mu,\nu) = \sup_{f \in \obs{X}} \int_X f d\mu - \int_X f d\nu.
    \]
    Now consider $\theta_1$. We have 
\begin{align*}
    \theta_1(\mu,\nu) &= \sup_{f \in \obs{X}} w_1(f_\# \mu, f_\# \nu)
    = \sup_{f \in \obs{X}} \sup_{g \in \obs{\real}} \int_\mathbb{R} g d f_\# \mu - \int_\mathbb{R} g d f_\# \nu \\
    &= \sup_{f \in \obs{X}} \sup_{g \in \obs{\real}} \int_X g \circ f d \mu  - \int_X g \circ f d \nu = \sup_{h \in \obs{X}} \int_X h d \mu  - \int_X h d \nu 
    = w_1(\mu,\nu),
\end{align*}
where we make the change of variables $h = g \circ f$ in the penultimate equality. 
\end{proof}

\begin{corollary}\label{cor:topologically_equivalent}
    For a compact space $X$, the metrics $\theta_p$ and $w_p$ are topologically equivalent for all $p \in [1,\infty)$.
\end{corollary}

\begin{proof}
    When $X$ is compact, the Wasserstein metrics $w_p$ are topologically equivalent to each other for all $p \in [1,\infty)$---see \cite[Corollary 6.13]{villani2009}. In particular, it is not hard to show that, for $1 \leq p \leq q < \infty$,
    \[
    w_q^q \leq D^{q-p} w_p^p,
    \]
    where $D$ is the diameter of $X$. For any $f \in \Lambda(X)$, we therefore have
    \[
    w_q(T_\mu (f), T_\nu (f))^q \leq D^{q-p} w_p(T_\mu(f),T_\nu(f))^p,
    \]
    where we consider $T_\mu(f)$ and $T_\nu(f)$ as measures on the compact image of $f$, which has diameter at most $D$, since $f$ is 1-Lipschitz. Since $f$ was arbitrary, we have 
    \[
    \theta_p(\mu,\nu) \leq \theta_q(\mu,\nu) \leq D^{1-p/q}\theta_p(\mu,\nu)^{p/q},
    \]
    as $\theta_p \leq \theta_q$ clearly holds. This proves topological equivalence of all $\theta_p$, so the 
    main claim follows from Proposition \ref{prop:equivalence_p_equals_1}. 
\end{proof}


\section{Discrete Model for Observable Distances}

This section develops a discrete model for measures supported on a fixed finite grid $A \subseteq X$ and provides the theoretical underpinning for its convergence properties.

\begin{definition} Let $(X,d)$ be a metric space and $\delta>0$. A subset $A \subseteq X$ is a {\em $\delta$-cover} of $X$ if $\bigcup_{a \in A} B(a, \delta) =X$. Equivalently, for any $x \in X$, there is $a \in A$ such that $d(x,a)<\delta$.
\end{definition}

Finite $\delta$-covers always exist if $X$ is compact. Indeed, the open cover of $X$ by all open balls $B(x,\delta)$, $x \in X$, admits a finite subcover $B(a_i,\delta)$, $1\leq i \leq n$. Thus, $A =\{a_1, \ldots, a_n\} \subseteq X$ is a $\delta$-cover.

\begin{lemma} \label{L:empiric}
If $A =\{a_1, \ldots, a_n\} \subseteq X$ is a finite $\delta$-cover of $X$, $\delta>0$, then there exists a Borel measurable map $p \colon X \rightarrow X$ such that $p(x) \in A$ and $d(x,p(x))< \delta$, $\forall x \in X$.
\end{lemma}

\begin{proof}
The argument is standard (cf.\,\cite[Lemma 2.2]{clement2008}). Starting with the finite open cover of $X$ given by the balls $B(a_i,\delta)$, $1 \leq i \leq n$, we construct a partition $X = S_1 \sqcup \cdots \sqcup S_n$, where each $S_i$ is Borel measurable and $S_i \subseteq B(a_i,\delta)$.  Set $S_1 = B(a_1, \delta)$. Inductively, assuming that $S_1, \ldots, S_{i-1}$ have been constructed, let $S_i = B(a_i, \delta) \setminus \bigcup_{j<i} S_j$. Clearly, $S_i \cap S_j = \emptyset$ if $i\ne j$, and $X = \bigsqcup_{i=1}^n S_i$.

Define $q \colon X \to A$ by $q(x) = a_i$ if $x \in S_i$. Since $S_i \subseteq B(a_i,\delta)$, we have that $d(x, q(x)) < \delta$, for any $x \in X$. Moreover, the pre-image under $q$ of any subset of $A$ is measurable because it is a union of elements of the partition $S_1, \ldots, S_n$. Then, the mapping $p \colon X \to X$ given by $p = \imath_A \circ q$ has the desired properties, where $\imath_A \colon A \hookrightarrow X$ is the inclusion map.
\end{proof}

\begin{proposition} \label{P:discrete}
Let $A = \{a_1, \ldots, a_n\} \subseteq X$ be a finite $\delta$-cover of $X$, $\delta>0$, and $p \geq 1$. Given $\mu \in W_p(X)$, there exists a Borel probability measure $\hat{\mu}$ whose support is contained in $A$ such that $w_p (\mu, \hat{\mu})< \delta$.
\end{proposition}

\begin{proof}
By Lemma \ref{L:empiric}, there exists a measurable map $p \colon X \rightarrow X$ such that $p(X) \subseteq A$ and $d(x,p(x))< \delta$, $\forall x \in X$. Let $\hat{\mu}= p_\sharp \mu$. By construction, the support of $\hat{\mu}$ is contained in $A$. To verify that $w_p(\mu,\hat{\mu}) < \delta$, let
$\varphi \colon X \to X \times X$ be given by $\varphi(x)=(x, p(x))$. Define $h= \varphi_\sharp \mu$, which gives a coupling between $\mu$ and $\hat{\mu}$. Indeed, let $\pi_1, \pi_2 \colon X \times X \to X$ denote the projections onto the first and second components, respectively. Then, $\pi_1 \circ \varphi = 1_X$ and $\pi_2 \circ \varphi = p$ so that
\begin{equation}
{\pi_1}_\sharp h = (\pi_1 \circ \varphi)_\sharp \mu = (1_X)_\sharp \mu = \mu
\quad \text{and} \quad
{\pi_2}_\sharp h = (\pi_2 \circ \varphi)_\sharp \mu = p_\sharp \mu = \hat{\mu},
\end{equation}
showing that $h \in \Gamma(\mu,\hat{\mu})$.
Therefore,
\begin{equation}
w_p (\mu, \hat{\mu}) \leq  
 \Big(\int_{X \times X} d^p(x,y) dh(x,y)\Big)^{1/p} = \Big(\int_{X} d^p(x,p(x)) d \mu (x)\Big)^{1/p} < \delta \Big(\int_{X} d \mu (x)\Big)^{1/p} = \delta,
\end{equation}
as claimed.
\end{proof}

\begin{corollary} \label{C:approx}
Let $0 \leq n < \infty$. Given $\mu,\nu \in W_p (X)$ and a finite $\delta$-cover $A \subseteq X$, $\delta>0$, there exist probability measures $\hat{\mu}, \hat{\nu} \in \borel{X}$ such that the supports of $\hat{\mu}$ and $\hat{\nu}$ are contained in $A$, $\theta_{p,n}(\mu,\hat{\mu})<\delta$, $\theta_{p,n}(\nu,\hat{\nu}) < \delta$, and $|\theta_{p,n} (\mu,\nu) - \theta_{p,n}(\hat{\mu},\hat{\nu})| \leq 2 \delta$.
\end{corollary}

\begin{proof}
By Proposition \ref{P:discrete}, there exist $\hat{\mu}, \hat{\nu} \in \borel{X}$ whose supports are contained in the $\delta$-cover $A$, $w_p (\mu, \hat{\mu})< \delta$ and $w_p (\nu, \hat{\nu})< \delta$. Since $\theta_{p,n} \leq w_p$, we have that
\begin{equation}
\begin{split}
\theta_{p,n}(\hat{\mu},\hat{\nu}) &\leq \theta_{p,n} (\hat{\mu},\mu) + \theta_{p,n}(\mu,\nu) + \theta_{p,n}(\nu, \hat{\nu}) \\
&\leq w_p(\hat{\mu},\mu) + \theta_{p,n}(\mu,\nu) + w_p(\nu,\hat{\nu}) \\
&\leq 2\delta + \theta_{p,n} (\mu,\nu).
\end{split}
\end{equation}
Similarly, $\theta_{p,n}(\mu,\nu) \leq 2\delta + \theta_{p,n}(\hat{\mu},\hat{\nu})$. Thus, $|\theta_{p,n}(\mu,\nu) - \theta_{p,n}(\hat{\mu},\hat{\nu})| \leq 2 \delta$.
\end{proof}

\begin{remark}
An identical argument shows that Corollary \ref{C:approx} also holds with $\theta_{p,n}$ replaced with $\theta_{p,\infty}$ or $\owp$.
\end{remark}

To obtain a fully discrete model, we also restrict the $n$-distance transform $\otr{\mu}^n \colon \Lambda_n (X) \to \borel{\real}$ to the subspace $\Lambda_n (X,A) \coloneqq \{f_a^\alpha \colon a \in A^{n+1}, \alpha \in I^{n+1}\} \subseteq \Lambda_n (X)$ of observables anchored on points in $A$. 

\begin{definition}
Let $A \subseteq X$ be a finite set, $0 \leq n <\infty$, and $\mu \in \borel{X}$. The {\em $n$-distance transform} $\atr{\mu} \colon \Lambda_n (X,A) \to \borel{\real} $ is defined by
\[
\atr{\mu} (f_a^\alpha) \coloneqq (f_a^\alpha)_\sharp \mu.
\]
\end{definition}

\begin{proposition} \label{P:ainjective}
Let $A \subseteq X$ be a finite set and $\mu, \nu \in \borel{X}$. If the supports of $\mu$ and $\nu$ are contained in $A$ and $T^0_{\mu,A} = T^0_{\nu,A}$, then $\mu=\nu$. A fortiori, for $n>0$, $\atr{\mu}=\atr{\nu}$ implies that $\mu=\nu$. 
\end{proposition}

\begin{proof}
The proof for $n=0$ is identical to that of Corollary \ref{C:oinjective}. The second statement follows from the fact that $\Lambda_0 (X,A) \subseteq \Lambda_n(X,A)$ for $n >0$.
\end{proof}

\begin{definition}\label{def:distance_transform_metric}
Let $A \subseteq X$ be a finite set and $n \geq 0$. Define the distance $\rdowp \colon W_p(X) \times W_p(X) \to \real$ by
\[
\rdowp(\mu,\nu) \coloneqq \sup_{f\in \Lambda_n(X,A)} w_p(\atr{\mu}(f), \atr{\nu} (f)) = \sup_{f\in \Lambda_n(X,A)} w_p (f_\sharp \mu, f_\sharp \nu).
\]
\end{definition}

\begin{proposition}
Let $A \subseteq X$ be a finite $\delta$-cover of $(X,d)$, $\delta>0$, and $\mu,\nu \in W_p(X)$. Then, there exist probability measures $\hat{\mu}, \hat{\nu} \in \borel{X}$ such that the supports of both $\hat{\mu}$ and $\hat{\nu}$ are contained in $A$, $\rdowp(\mu,\hat{\mu})<\delta$, $\rdowp(\nu,\hat{\nu}) < \delta$ and $|\rdowp (\mu,\nu) - \rdowp(\hat{\mu},\hat{\nu})| \leq 2 \delta$, for any $n \geq 0$.
\end{proposition}

\begin{proof}
The proof is similar to that of Corollary \ref{C:approx} that guarantees the existence of $\hat{\mu}, \hat{\nu} \in \borel{X}$ supported in $A$ such that $\theta_{p,n}(\mu,\hat{\mu})<\delta$, $\theta_{p,n}(\nu,\hat{\nu}) < \delta$, and $|\theta_{p,n}(\mu,\nu) - \theta_{p,n}(\hat{\mu},\hat{\nu})| \leq 2 \delta$. This implies that $\rdowp (\mu, \hat{\mu}) \leq \theta_{p,n}(\mu,\hat{\mu})<\delta$ and $\rdowp (\nu, \hat{\nu}) \leq \theta_{p,n}(\nu,\hat{\nu})<\delta$. Therefore,
\begin{equation}
\rdowp(\hat{\mu},\hat{\nu}) \leq \rdowp(\hat{\mu},\mu) + \rdowp(\mu,\nu) +\rdowp(\nu, \hat{\nu}) \leq 2\delta + \rdowp(\mu,\nu).
\end{equation}
Similarly, $\rdowp(\mu,\nu) \leq 2 \delta + \rdowp(\hat{\mu},\hat{\nu})$. This proves the proposition.
\end{proof}

\section{Numerical Experiments} \label{S:numexp}

A major benefit of the observable Wasserstein (OW) framework is that it is extremely simple to implement (empirical approximations of) the lower bound $\theta_{p,n}$. In this section, we present several numerical experiments, which explore the behavior of the distance, in comparison to related metrics such as the sliced Wasserstein (SW) and Chamfer distances. The experiments are mostly proof-of-concept in nature, intended to illustrate the properties of observable Wasserstein distances. The code for the experiments is available on our GitHub repository\footnote{https://github.com/trneedham/Observable-Wasserstein-Distance}. 


\subsection{Comparison to Sliced Wasserstein Distance: Gaussian Classification}\label{sec:Gaussian_comparison}

The first synthetic experiment compares the performance of OW distances to SW distances in a simple classification task. Our data is generated as follows. Fixing a dimension $d$, we define covariance matrices $\Sigma_i$, $i \in \{1,2,3\}$ by setting $\Sigma_1 = \mathrm{I}_d$ to be the $d \times d$ identity matrix, $\Sigma_2$ to be the diagonal matrix with diagonal entries $(3,1,1,\ldots,1)$, and $\Sigma_3$ to be the diagonal matrix with diagonal entries $(1,1,\ldots,1,3)$. These covariance matrices are used to distinguish three classes in our constructed dataset. To generate a sample from class $i$, we sample $250$ points from the Gaussian distribution $\mathcal{N}(0,\Sigma_i)$. The dataset for dimension $d$ consists of $10$ samples of this form for each of the three classes. 

The task in this experiment is to use a chosen optimal transport distance to distinguish the classes. First, we use max SW distance with $p=1$, for each number of slices $n \in \{10,20,30,40,50\}$. Slice directions are chosen uniformly at random; we sample directions 10 times for each run and report average results. Next, we use OW distances in the distance transform formulation of Definition~\ref{def:distance_transform_metric}, with $p=1$. This is achieved by sampling $n \in \{10,20,30,40,50\}$ points from a Gaussian $\mathcal{N}(0,5 \cdot \mathrm{I}_{d})$, and using these as the set of anchor points $A$ to measure distances against. In our experiments, we sample anchor points $10$ times for each run, and average the results. We found performance to be relatively stable under choice of $A$. The choice of width of the Gaussian (i.e., size of the covariance matrix) used for sampling anchors---in particular, that it was larger than the widths of the Gaussians used to generate the dataset---was found to be helpful in distinguishing the classes.

For each choice of OT distance, we measure distinguishing performance by 1-nearest neighbor classification score. We compute classification scores for datasets in dimensions $d \in \{2,5,10,25,50,75,100\}$; for each dimension, we run the whole experiment 10 times and report the average classification score for each method. Results of this experiment are reported in Figure~\ref{fig:GaussianClassification}, as well as average compute times for each method. The takeaway  is that the performances of SW and OW distances are similar---classification scores are comparable for each method, across number of slices, with OW gaining an edge over SW as dimension increases, perhaps due to measure concentration effects in the SW projections.

\begin{figure}
    \centering
    \includegraphics[width=1\linewidth]{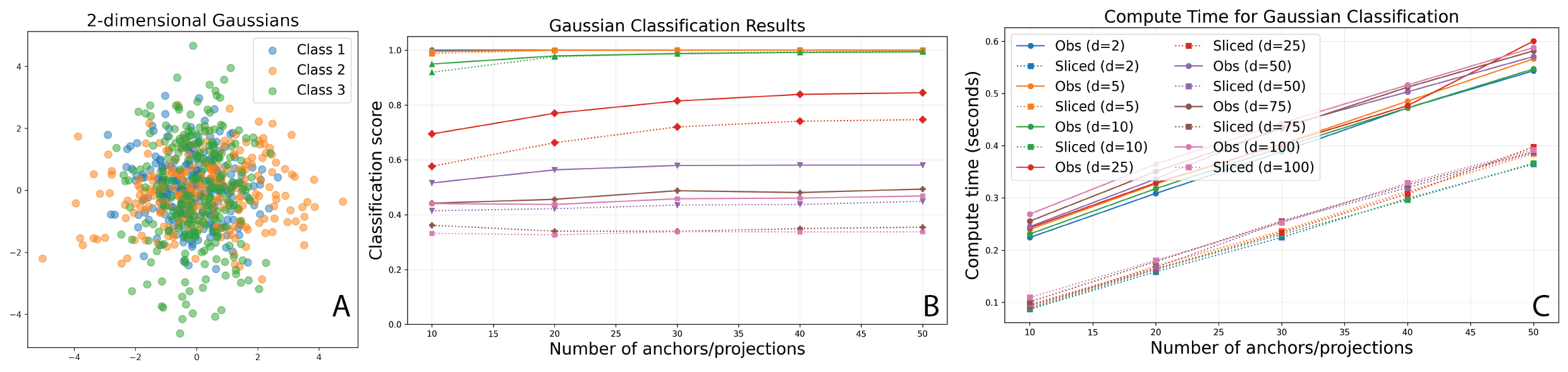}
    \caption{Classifying Gaussian measures with Sliced and observable Wasserstein distances (see Section \ref{sec:Gaussian_comparison}). {\bf A:} Samples from the 2-dimensional Gaussian datasets. {\bf B:} Nearest neighbor classification scores for various optimal transport metrics, plotted against number of projections or anchor points used in their computations. The legend from {\bf C} applies to the plot in {\bf B}. {\bf C:} Compute times for all methods used in the Gaussian classification experiment.}
    \label{fig:GaussianClassification}
\end{figure}

\subsection{Comparison to Wasserstein Distances: Distributions on Graphs}\label{sec:graph_classification}

This experiment treats distributions on non-Euclidean spaces, so that the traditional sliced Wasserstein distance is not applicable, and we therefore compare to classical Wasserstein distance. This is approached through another simple classification task for distributions on graphs. To construct our dataset, we first draw a random graph $G$ via the \texttt{networkx} implementation of the random geometric graph model~\cite{hagberg2008exploring,penrose2003random}: we sample $n$ nodes uniformly at random from a unit square $[0,1] \times [0,1]$, and then connect any pair of nodes at distance less than a given radius $r$. In our experiments, we set $r = \sqrt{\log (n)/n}$, which produces relatively sparse graphs that are still likely to be connected; disconnected graphs are rejected and we draw again until a connected graph is constructed. Once an admissible graph $G$ has been generated, we define a random distribution from one of three classes. The class is determined by a \emph{location} parameter---the unit square is divided into 9 regions by a uniform $3 \times 3$ grid, and the classes correspond to the \emph{top-left}, \emph{middle}, and \emph{bottom-right} regions of the graph. Given a choice of location, a node is drawn from the corresponding grid square uniformly at random. We then diffuse heat from this node via the combinatorial Laplacian, for time $t$; in our experiments, we used $t = 0.1 \cdot n$, as this tends to yield distributions which, qualitatively, are neither highly concentrated nor highly diffuse. Next, we add random noise to the resulting heat function, depending on a \emph{noise level} parameter---for noise level $\beta$, random noise of scale $\beta$ times the maximum of the current heat function is added. Finally the random distribution is obtained by truncating negative values and normalizing. See Figure \ref{fig:GraphClassification}A for examples of these distributions.

The task in this experiment is to distinguish classes of distributions over a given graph via various OT distances. To treat the graph $G$ as a metric space, we use shortest path distance, with unweighted edges. Here (as sliced Wasserstein is no longer applicable), we use Wasserstein distance as a baseline. Specifically, we use the \emph{Earth Mover's Distance} implementation of Wasserstein distance (i.e., with $p=1$) from the \texttt{python optimal transport} package. We also implement the distance transform version of observable Wasserstein distance, with $p=1$, by taking our set $A$ to be a random sample of nodes from the graph; the number of nodes is taken to be a percentage of the total nodes, from the set $\{5\%, 10\%, 15\%\}$. 

Given one of the OT distances described above, nearest neighbor classification scores for each dataset are computed, as in Section 6.1 of the main document. A single run of an experiment consists of a choice of random graph with $n\in \{300,400,500,600,700\}$ nodes, a choice of noise level $\beta \in \{0.5,1.0,1.5,2.0,2.5,3.0\}$, from which a dataset of 10 distributions from each class is generated on $G$, and classification scores are computed. For each choice of $n$ and $\beta$, 10 runs of the experiment are performed, and average results are reported.

The results of this experiment are shown in Figure \ref{fig:GraphClassification}. We note that the observable Wasserstein distances outperform Wasserstein distance in classification score. In addition, the computation time for Wasserstein distances increases much more rapidly than those of the observable Wasserstein distances as the number of nodes increases. One potential explanation for this performance is that the observable Wasserstein distance could be more robust to noise than the classical Wasserstein distance. This conjecture is supported by the fact that improved robustness of max sliced Wasserstein distance has been theoretically demonstrated in~\cite[Theorem 2]{nietert2022statistical}, where it is shown to enjoy dimension-free risk bounds in an outlier noise contamination model, avoiding the dimension factors that appear in classical Wasserstein distance~\cite{nietert2022outlier}.

\begin{figure}
    \centering
    \includegraphics[width=1\linewidth]{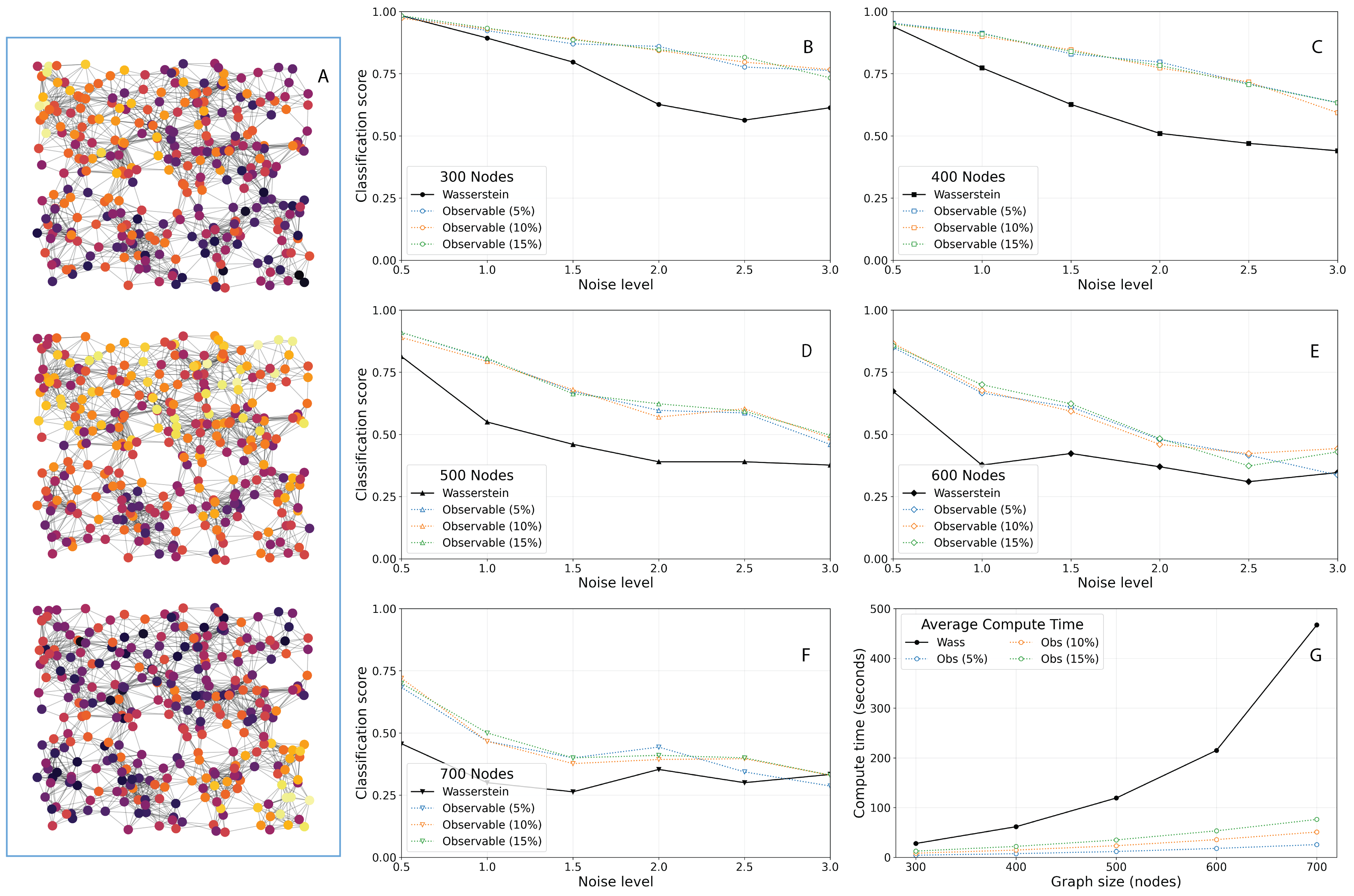}
    \caption{Classifying distributions on graphs with Wasserstein and observable Wasserstein distances (see Section \ref{sec:graph_classification}). {\bf A:} Examples of distributions from three classes (top-left, middle, bottom-right, respectively) on a fixed graph $G$, with $n=300$ nodes, noise level $\beta = 1.5$. {\bf B--F:} Classification scores for graphs with $n \in \{300,400,500,600,700\}$ nodes, respectively. {\bf G:} Average compute times of each method (across all noise levels) for graphs of various sizes.}
    \label{fig:GraphClassification}
\end{figure}

\subsection{Dependence on Type of Observables: Distributions on Spheres}\label{sec:distributions_on_spheres}

The experiments in this subsection are designed to determine the effect of using observables from the set $\Lambda_n(X)$, versus simple distance-to-a-point functions (as were used exclusively in the experiments above). Here, we fix the ambient metric space $(X,d)$ to be a $d$-dimensional unit sphere, endowed with geodesic distance. Given $d$ and a number of samples $m \in \{100,200\}$, we generate two discrete distributions $\mu$ and $\nu$ by uniformly sampling $m$ points from the $d$-sphere. Our goal is to understand how empirical estimates of the OW distance between these distributions compare to the Wasserstein $1$-distance under various choices of observables. 

We choose a maximum number of functions to include in an observable $n_f \in \{1,3,5,7,9\}$, as well as a number of observables $n_o \in \{40,80,120,160,200\}$. We sample $1000$ points from the unit sphere as a collection of potential anchor points. To estimate the OW $1$-distance, we  randomly choose $n_o$ collections of anchor points of the form $a = \{a_1,\ldots,a_{n_f}\}$. For each such collection, we include observable functions of the form $f_{a_{i_1}} \wedge f_{a_{i_2}} \wedge \cdots \wedge f_{a_{i_k}}$, for all subsets  $\{a_{i_1},\ldots,a_{i_k}\} \subset a$. The OW distance is then estimated via this collection of observables.

For each choice of dimension $d$, number of samples $m$, number of functions $n_f$ and number of observables $n_o$, we calculate the \emph{relative error} $
(\mbox{W} - \mbox{ OW})/\mbox{W}$, 
where ``W'' is Wasserstein $1$-distance between $\mu$ and $\nu$ and ``OW'' is the OW distance estimate described above. This is repeated 10 times (redrawing $\mu$ and $\nu$ each time), and the average result is recorded. 

The results of the experiments are shown in Figure \ref{fig:sphereExperiment}. The figure shows the trade-off between relative error for the various parameters with compute time per instance of the distance. As either number of observables $n_o$ or number of functions $n_f$ increases, the relative error decreases, as these increases yield better estimates of OW distance.
We observe that increasing the number of functions $n_f$ yields diminishing returns in lowering the relative error, while incurring significant computational cost, once it is increased from $n_f = 7$ to $n_f = 9$.

\begin{figure}
    \centering
    \includegraphics[width=0.98\linewidth]{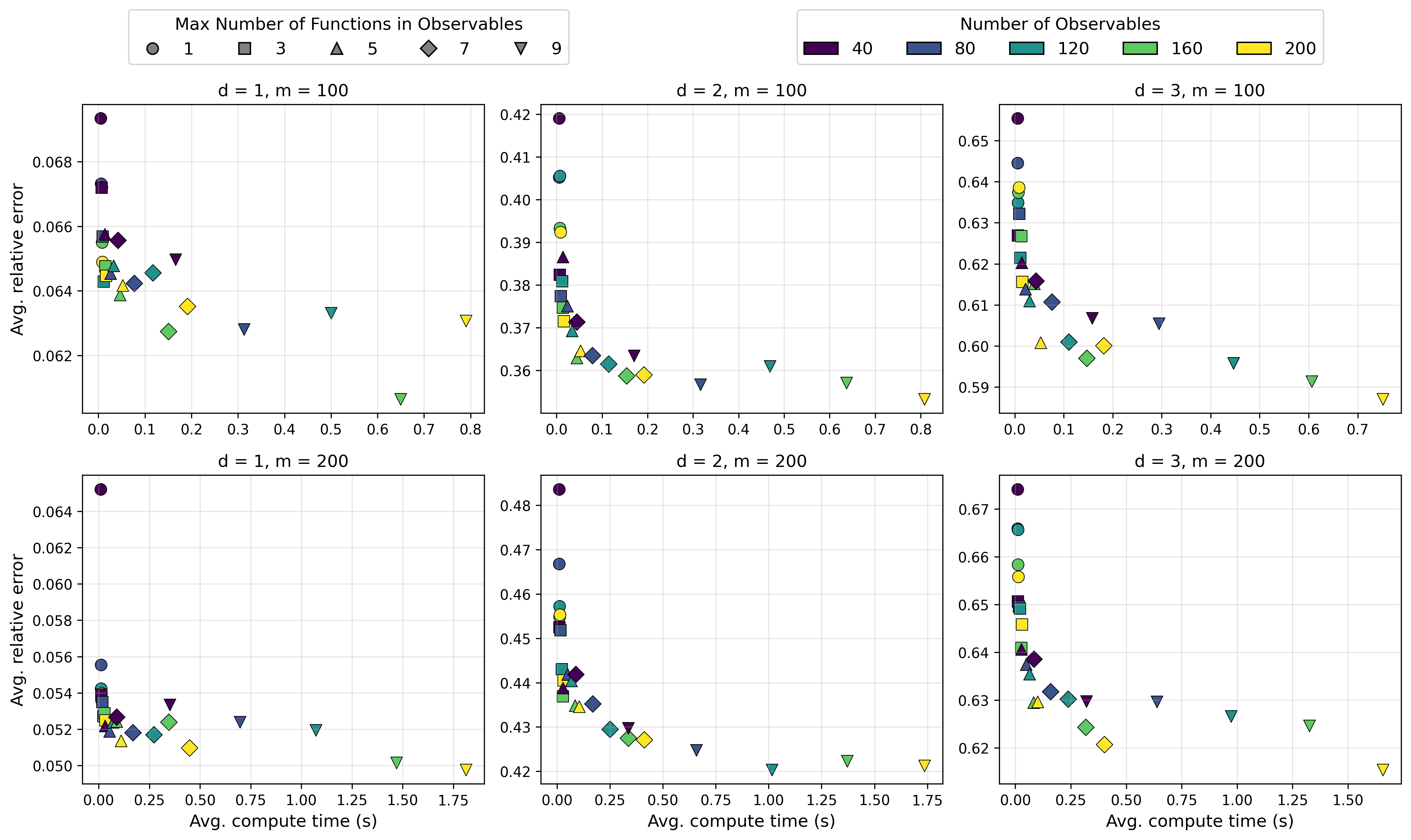}
    \caption{Relative errors of estimated observable Wasserstein distances compared to compute time (see Section \ref{sec:distributions_on_spheres}). The legend applies to all plots, with point shape corresponding to number of functions in observables ($n_f$ in the text) and color corresponding to number of observables ($n_o$ in the text). Each plot gives results for a combination of sphere dimension $d$ and number of points in the random distribution $m$.}
    \label{fig:sphereExperiment}
\end{figure}

\subsection{Classification Performance on Real Data}\label{sec:modelnet}

In this experiment, we compare the classification performance of scalable point cloud metrics on the \texttt{ModelNet10} dataset~\cite{wu20153d}, which consists of 10 classes of meshes of everyday objects such as bathtubs, beds and chairs. In the experiment, we use a small version of the dataset, consisting of 1308 meshes total, roughly balanced across the 10 classes. We treat each mesh as a point cloud in $\mathbb{R}^3$ by sampling 1024 points from each surface, then normalize so that each point cloud lies in a unit ball. The dataset is divided into a training set of 400 point clouds and a testing set of 908 point clouds. We then create noisy versions of this dataset by adding $1024 \cdot \eta$ random Gaussian (mean zero, standard deviation $\sigma = 2$) points to each point cloud, with $\eta \in \{0,0.2,0.4,0.6,0.8,1.0\}$. Examples of point cloud data in the experiment are shown in Figure \ref{fig:modelnet}A.

For each noise level $\eta$, we quantify the classification power of a given  point cloud metric $d$ via 1- and 5-nearest neighbor classification scores (using the distance from a test point cloud to its nearest training point clouds). The metrics we consider are the max-SW distance (with 100 slice directions sampled from the uniform distribution on the sphere), the OW distance using only distance-to-a-point functions (with 100 anchor points sampled randomly from a mean zero, standard deviation $\sigma = 2$, Gaussian distribution), and the OW distance using observables of the form $f_{a_1} \wedge \cdots \wedge f_{a_5}$ (100 such observables, where anchor points are randomly sampled, as above). Note that we do not use observables with all possible subsets of anchor points, as in the previous subsection, for the sake of keeping the compute time feasible. As an additional baseline, we use the \emph{Chamfer distance}~\cite{barrow1977parametric,borgefors1988hierarchical}; the Chamfer distance between point clouds $P$ and $Q$ is given by
\begin{equation}\label{eqn:chamfer}
\sum_{p \in P} \min_{q \in Q} \|p-q\|^2 + \sum_{q \in Q} \min_{p \in P} \|p-q\|^2.
\end{equation}
This baseline is included, due to its ubiquity in machine learning on point clouds~\cite{fan2017point,yuan2018pcn,wu2021density}. 

Results of the experiment are reported in Figure \ref{fig:modelnet}. We observe that, once noise is added to the dataset, the OW distances provide the best classification scores across all metrics. Moreover, its compute time scales similarly to SW distance, and more tractably than Chamfer distance---we note that we are using a straightforward implementation of Chamfer distance, so this scaling issue could potentially be mitigated by applying a more sophisticated algorithm~\cite{bakshi2023near}.

\begin{figure}
\includegraphics[width=0.95\linewidth]{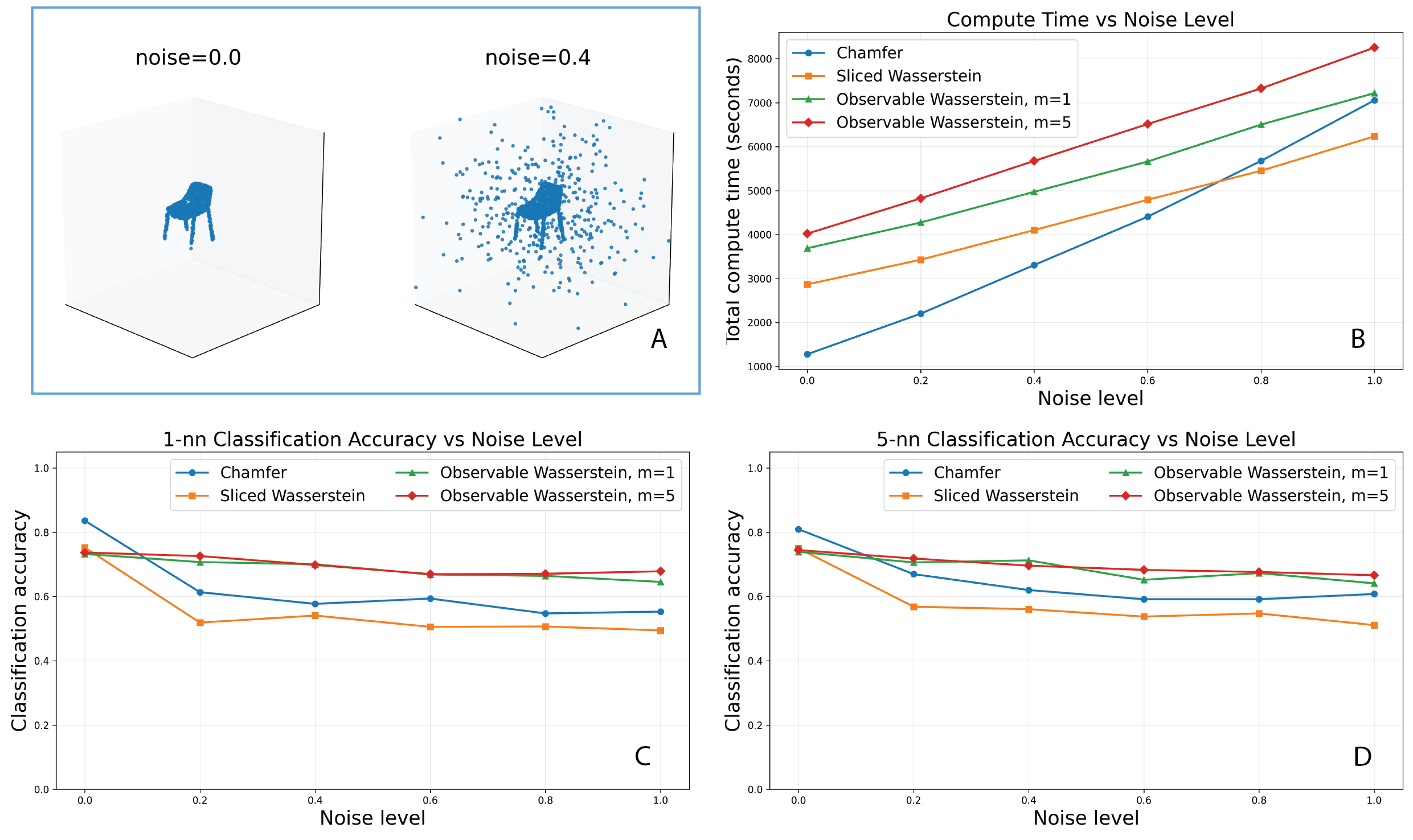}
\caption{Nearest neighbor classification on \texttt{ModelNet10} via scalable point cloud distances (see Section \ref{sec:modelnet}). {\bf A:} An example clean and noisy point cloud from the \texttt{ModelNet10} dataset. {\bf B:} Compute times for the full experiment for each metric. {\bf C:} 1-Nearest Neighbor classification scores for each distance across noise levels. {\bf D:} 5-Nearest Neighbor classification scores.}\label{fig:modelnet}
\end{figure}

\subsection{Application to Deep Learning: Point Cloud Autoencoder Loss}\label{sec:autoencoder}
In this proof-of-concept experiment, we implement a simple point cloud autoencoder, using (a variant of) OW distance in the loss function. When learning representations of point clouds, it is common in the literature to use a loss based on the Chamfer distance~\cite{hermosilla2019total,duan20193d,deng2018ppf}. Previous works have also utilized 1-Wasserstein~\cite{achlioptas2018learning} and SW distances~\cite{nguyen2021point}, which have been shown to have better performance than the Chamfer distance in classification and registration tasks. The framework proposed here also compares to the Chamfer distance, while a more thorough implementation and comparison to other variants is left as follow-up work.

The architecture of our autoencoder is summarized as follows:

\begin{itemize}
    \item \textbf{Input:} The model takes as input a point cloud $X \subset \mathbb{R}^2$ of fixed size $N$ (in our experiments, $N=128$). 

    \item \textbf{Encoder:} The encoder is a PointNet-style~\cite{qi2017pointnet} map
    \[
    E \colon (\mathbb{R}^2)^N \to \mathbb{R}^m,
    \]
    where \(m\) is the latent dimension, which was fixed in our experiments as $m=128$. First, a common multilayer perceptron (MLP), with ReLU activation, is applied to each point, with layer dimensions 
    \[
    2 \to 64 \to 128 \to 256 \to 512.
    \]
    Next, max pooling is applied, to enforce permutation invariance. Finally another MLP is applied, with layer dimensions 
    \[
    512 \to 256 \to m.
    \]
    The resulting vector in $\mathbb{R}^m$ is the latent representation of the point cloud.

    \item \textbf{Decoder:} The decoder is an MLP
    \[
    D \colon \mathbb{R}^m \to (\mathbb{R}^2)^N
    \]
    with layer dimensions
    \[
    m \to 256 \to 512 \to 1024 \to 2N.
    \]

    \item \textbf{Loss Function.} Let $\omega$ denote the weights for a particular model, with associated encoder $E_\omega$ and decoder $D_\omega$. For a training dataset $\{X_i\}_{i=1}^M$, set $X^\omega_i = D_\omega \circ E_\omega(X_i)$. The loss function is based on a convex combination of Chamfer and  observable Wasserstein distances:
    \[
    \mathcal{L}(\omega) = \sum_{i=1}^M (1-\alpha) L_{Ch}(X_i,X_i^\omega) + \alpha L_{Ob}(X_i,X_i^\omega),
    \]
    where $\alpha \in [0,1]$. The terms in the sum are given by 
    \begin{itemize}
        \item $L_{Ch}(P,Q)$ is squared Chamfer distance between $P$ and $Q$.
        \item $L_{Ob}(P,Q)$ is the square of an empirical version of an $L^2$-style observable \\ Wasserstein distance. We uniformly sample $k$ anchor points $a_i \in [0,1] \times [0,1]$ (with $k=64$ in our experiments) and compute 
        \[
        L_{Ob}(P,Q) = \sum_{j=1}^k W_1((f_{a_j})_\# P, (f_{a_j})_\# Q)^2,
        \]
        where $f_{a_j}(x) = \|a_j - x\|$. We abuse notation and consider $P$ and $Q$ as empirical distributions. While the majority of the paper focused on $L^\infty$-type formulations of observable Wasserstein distance, we opted for the $L^2$ version here for improved differentiability properties when  training the models. 
    \end{itemize}
\end{itemize}

We apply our autoencoder pipeline to the \textsc{MNIST} dataset of handwritten digits, which we convert to point-cloud data as follows. Given a digit (a $28 \times 28$ grayscale image), we choose a fixed intensity threshold (we used threshold $0.25$ in our experiments) and retain only those pixels whose grayscale value is high enough, each of which is treated as a point in the point cloud, considered as a subset of the unit square $[0,1] \times [0,1]$. Next, we normalize the resulting point cloud to have a fixed number of points $N$ (we used $N=128$ points in our experiments) by either randomly subsampling or randomly duplicating points.

Additionally, we generate a noisy version of the MNIST dataset as follows. For a fixed noise level $\rho > 0$, we add $\lfloor \rho \cdot N \rfloor$ noise points, uniformly sampled from $[0,1] \times [0,1]$, then subsample the result to once again have $N$ points. In this setting the autoencoders are trained as \emph{de-noisers}: the input is a noisy point cloud, which is encoded and decoded, with the output compared against the ground truth clean point cloud via the loss function.

In our experiments, we trained autoencoders with loss function weights $\alpha \in   \{0,0.25,0.5,$ $0.75,1\}$. In each case, training was run for 50 epochs on the full MNIST training data (50,000 images) using batches of size 128. This was performed on the clean MNIST data as well as the noisy MNIST data in the context of de-noising. 

Figure~\ref{fig:autoencoder_results} shows some example reconstructions for the pure Chamfer, pure observable Wasserstein and mixed ($\alpha = 0.5$) loss functions. We observe that Chamfer decodings yield sharper renderings of the point cloud, but exhibit uneven point density. In contrast, observable Wasserstein decodings, while introducing some outlier noise, are generally more uniform in density. In the noisy regime, the Observable loss is able to recover digit features which are destroyed by the Chamfer loss. Averages of digit classes in the testing datasets for the Chamfer and observable Wasserstein models are shown in Figure~\ref{fig:LatentAverages}. The averages are computed by embedding the datasets in the latent space, computing Euclidean averages of the latent codes, then decoding. Here, we once again observe a tradeoff between sharpness and uneven density in the decodings between the two models.

\begin{figure}
    \centering
    \includegraphics[width=1\linewidth]{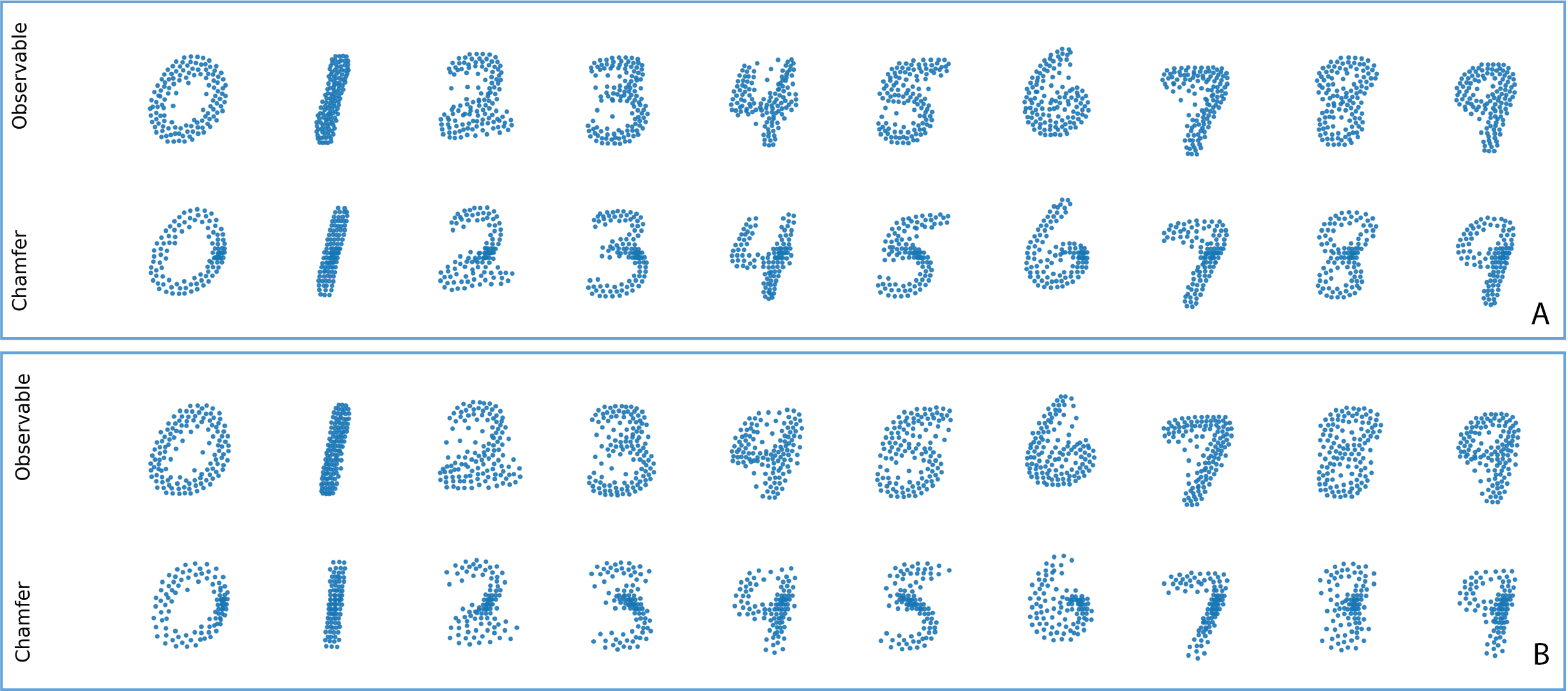}
    \caption{Averages of MNIST test digit classes from latent space codes from the autoencoders using pure Chamfer and pure observable Wasserstein losses, respectively. {\bf A:} Averages from the clean MNIST dataset. {\bf B:} Averages from the noisy MNIST model.}
    \label{fig:LatentAverages}
\end{figure}

A more quantitative summary of performance is presented in Table \ref{tab:mnist-latent-logreg}. Here we run a logistic regression classifier on the embeddings of the datasets into the latent space. The classifier is trained on the embedding of the training set in each case, and evaluated on the embedding of the testing set (consisting of 10,000 digits). We observe that classification accuracy is highest when pure Observable loss is used. We emphasize here that our deep learning pipeline was not trained as a classifier, and these classification results are only meant to quantify how well classes are separated in the latent space.

\begin{figure}
    \centering
    \includegraphics[width=1\linewidth]{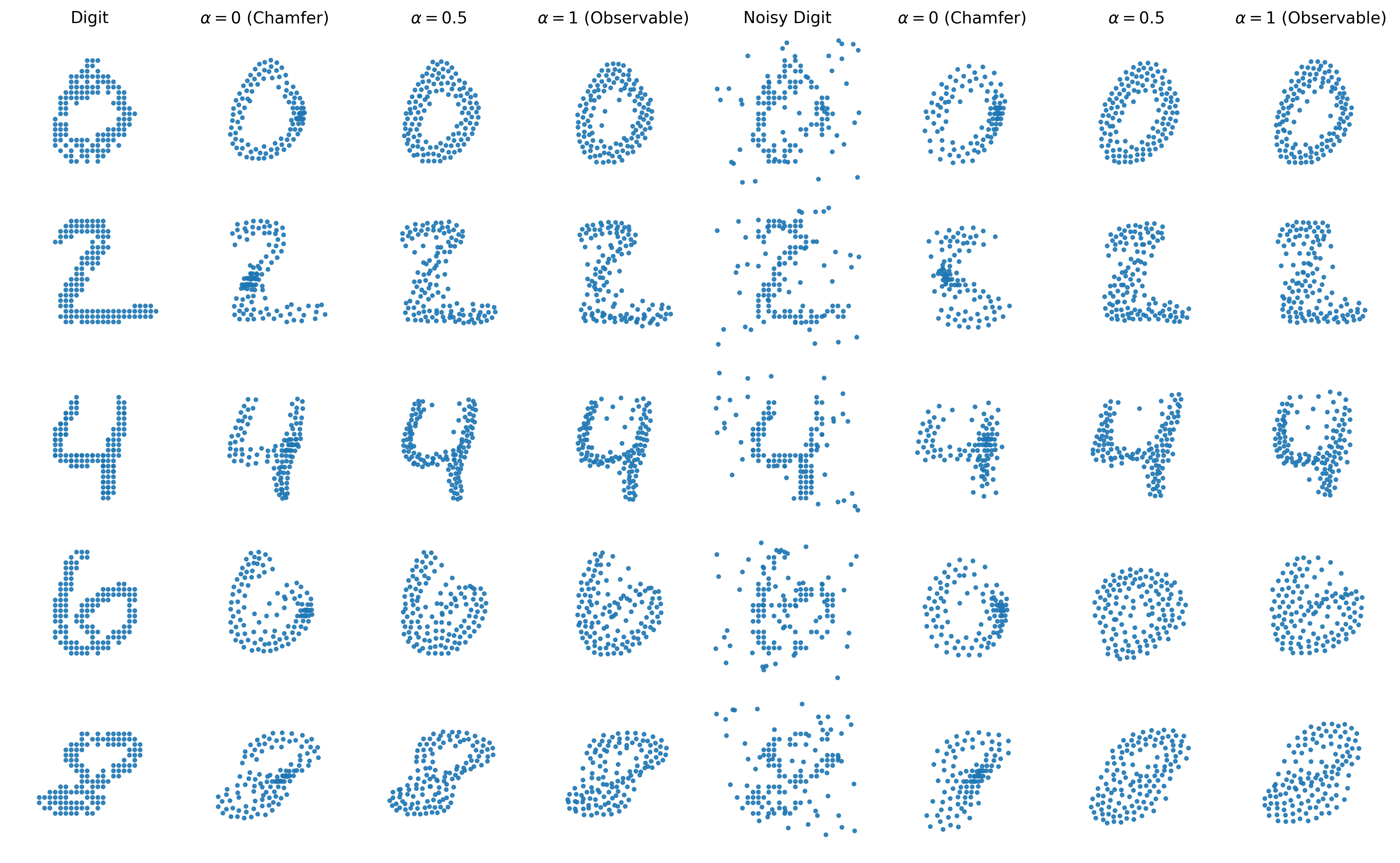}
    \caption{Autoencoder results on the MNIST dataset. The left column shows some example digits from the testing set. Subsequent columns show decoding results from the autoencoders with $\alpha = 0$ (corresponding to pure Chamfer loss), $\alpha = 0.5$, and $\alpha = 1$ (corresponding to pure observable Wasserstein loss). Column 5 shows examples from the noisy digit testing set, and subsequent columns show autoencoder reconstructions.}
    \label{fig:autoencoder_results}
\end{figure}

\begin{table}[t]
\centering
\begin{tabular}{lccccc}
\hline
Dataset
& $\alpha=0$ (Chamfer)
& $\alpha=0.25$
& $\alpha=0.5$
& $\alpha=0.75$
& $\alpha=1$ (Observable) \\
\hline
MNIST
& 87.31
& 88.27
& 88.78
& 89.64
& \textbf{90.92} \\
Noisy MNIST
& 72.36
& 73.74
& 77.95
& 78.56
& \textbf{80.67} \\
\hline
\end{tabular}
\caption{Latent-space logistic regression accuracy for MNIST point-cloud autoencoders.}
\label{tab:mnist-latent-logreg}
\end{table}

\section{Summary and Discussion} \label{S:discussion}

In this paper, we introduced the observable Wasserstein distance, and several computationally tractable lower bounds. These lower bounds were shown to distinguish measures belonging to certain natural classes characterized by the dimension of their supports, on which they define true metrics. We demonstrated the efficiency and favorable performance of our framework in several proof-of-concept numerical experiments. 

The ideas presented here raise several questions that we plan to address in future work:

\begin{itemize}
    \item Proposition \ref{prop:equivalence_p_equals_1} shows that $\theta_1 = w_1$. For $p \neq 1$, one might ask whether $\theta_p$ and $w_p$ are bi-Lipschitz equivalent,  $w_p \leq L \theta_p$,  and to understand the dependence of the Lipschitz constant $L$ on the geometry of $X$. This problem parallels existing work on the analogous question for sliced Wasserstein distances~\cite{carlier2025sharp,bayraktar2021strong,bonnotte2013unidimensional}.
    \item The sampling convergence and robustness properties of sliced Wasserstein distance are well understood (cf.~\cite{nietert2022statistical,lin2021projection,manole2022minimax}). It would be interesting to extend these results to the setting of observable Wasserstein distances.
    \item The computationally tractable approach to observable Wasserstein distance is through approximations of $\theta_{p,n}$, which involves restricting observables to a class derived from distance-to-a-point functions. This is a natural class of Lipschitz functions to consider, but other sets of observables can be potentially more expressive.
    \item The numerical experiments in this paper were mainly proof-of-concept in nature. It remains an open task to develop a fully fleshed out numerical framework, and to run more in-depth experiments, involving a wider variety of ambient metric spaces. Of special interest is the development of the deep learning angle, for example, to leverage the flexibility of the framework by learning useful observables as part of training.  
\end{itemize}

\section*{Declarations}

\noindent {\bf Conflict of Interest } The authors declared that they have no conflict of interest. \\

\noindent {\bf Funding statement} 
Part of this work was completed while WM was a visitor at Universidade Federal de S\~{a}o Carlos (UFSCar), Brazil, supported in part by FAPESP grant 2022/16455-6. TN was supported by NSF grants DMS 2324962 and CIF 2526630.  ES was supported by FAPESP, Brazil, grant 2022/16455-6. LM was supported by FAPESP, Brazil, grant 2024/14246-6.\\

\noindent {\bf Author Contributions} These authors contributed equally to this work.\\

\noindent {\bf Acknowledgments} The authors thank Ece Karacam for contributing some illustrations to the paper.  The authors acknowledge the use of ChatGPT~5.5 for assistance with coding aspects of the numerical experiments.


\bibliography{owbiblioSpringer}

\end{document}